\documentclass[11pt,a4paper]{article}

\usepackage{lmodern}
\usepackage{amsmath,amssymb,amsthm,mathtools}
\usepackage{geometry}
\usepackage{xcolor}            
\usepackage[normalem]{ulem}    
\usepackage{hyperref}
\hypersetup{colorlinks=true, linkcolor=blue, citecolor=blue}
\usepackage{enumitem}
\usepackage{booktabs}
\usepackage{authblk}

\mathtoolsset{showonlyrefs}

\newtheorem{theorem}{Theorem}[section]
\newtheorem{proposition}[theorem]{Proposition}
\newtheorem{lemma}[theorem]{Lemma}
\newtheorem{corollary}[theorem]{Corollary}
\theoremstyle{definition}
\newtheorem{definition}[theorem]{Definition}
\newtheorem{remark}[theorem]{Remark}
\newtheorem{example}[theorem]{Example}

\newcommand{\R}{\mathbb{R}}

\newcommand{\cL}{\mathcal{L}}

\begin{document}

\title{\bfseries The Financial Bubble Model\\ with L\'evy Jump Processes\thanks{R. Barkhudaryan and V. Khalatyan were supported by the Higher Education and Science Committee of the MESCS RA under Research Project No.~25RL-1A040.}}

\author[2,1,3]{Avetik Arakelyan\thanks{ \texttt{email: avetik.arakelyan@ysu.am}}}
\author[1,2]{Rafayel Barkhudaryan\thanks{\texttt{email: barkhudaryan@ysu.am}}}
\author[1]{Vigen Khalatyan\thanks{\texttt{email: v.khalatyan@ysu.am}}}
\author[1]{Michael Poghosyan\thanks{\texttt{email: michael@ysu.am}}}

\affil[1]{Yerevan State University, Yerevan, Armenia}
\affil[2]{Institute of Mathematics of NAS RA, Yerevan, Armenia}
\affil[3]{Center for Scientific Innovation and Education, Yerevan, Armenia}

\date{}

\maketitle

\begin{abstract}

In this work we consider an extension of the Berestycki--Monneau--Scheinkman (BMS) model \cite{BMS14} for speculative financial bubbles, in which the investor disagreement process is allowed to have jumps. While the original BMS framework assumes that disagreement evolves along continuous paths, our model accounts for sudden shifts in market sentiment through an independent L\'evy jump process. Using optimal stopping theory and the It\^o--L\'evy formula, we show that the speculative bubble premium satisfies a non-local partial integro-differential equation (PIDE) with a moving obstacle. We develop a viscosity solution theory for this non-local obstacle problem. We prove a comparison principle by a doubling-of-variables argument adapted to the non-local jump integral, and we prove existence and uniqueness of the bubble price by Perron's method, constructing explicit continuous sub- and supersolutions. We then introduce a monotone Implicit--Explicit finite difference scheme for the bubble premium. Following the Barles--Souganidis framework, we show that the discrete operator preserves the M-matrix property and that the scheme converges locally uniformly to the unique viscosity solution under a state-dependent Courant--Friedrichs--Lewy (CFL) condition. At the end of the paper we implement the scheme via a PSOR--Picard algorithm and present numerical tests for four finite- and infinite-activity L\'evy models.
\end{abstract}

\vspace{1em}

\section{Introduction}

Speculative asset bubbles, from the 17th-century Tulip Mania to the dot-com boom and the more recent cryptocurrency episodes, have repeatedly been associated with macroeconomic crises. The behavioral origins of these phenomena are well documented in the economic literature \cite{Kindleberger2005,Shiller2015}, but describing the mechanics of bubble formation within a mathematical framework remains a difficult problem.

One line of research, due to Jarrow, Protter, and their co-authors \cite{JarrowProtter2007,JarrowProtter2010}, works in a classical arbitrage-free setting. Here a bubble is defined as the difference between the market price and the fundamental value, and it is described mathematically as a strict local martingale under the risk-neutral measure. Such models are well suited for statistical detection and for arbitrage-free derivative pricing, but they mainly describe the explosion of the price path rather than the mechanism that generates the overvaluation.

To describe the formation of bubbles, we follow instead the equilibrium approach, which is based on investor disagreement and on limits to arbitrage. It was introduced by Miller \cite{Miller1977} and formalized by Harrison and Kreps \cite{HarrisonKreps1978}, and it explains a bubble through the so-called resale option. When short-selling is constrained, the owner of an asset is willing to pay a premium above the fundamental dividend value, because he holds an American-style option to resell the asset to a more optimistic investor in the future. This mechanism was carried over to continuous-time diffusions by Scheinkman and Xiong \cite{ScheinkmanXiong2003}, and was further studied by Chen and Kohn \cite{ChenKohn2011}. Berestycki, Monneau, and Scheinkman (BMS) \cite{BMS14} reduced this trading mechanism to a macroscopic model and showed that the bubble value function satisfies a parabolic obstacle partial differential equation (PDE), whose free boundary is the optimal trading threshold between the two groups of agents. For the numerical treatment of such problems we refer to \cite{Arakelyan2019,MR3552318}.

The BMS framework and its extensions assume that the disagreement process evolves along continuous diffusion paths. Continuous paths, however, do not capture the sudden shocks that are present in financial markets. Empirical studies indicate that market sentiment, liquidity crises, and macroeconomic news arrive as abrupt jumps, which produce heavy-tailed return distributions that cannot be obtained from a pure diffusion \cite{ContTankov2004}.

In this paper we generalize the heterogeneous beliefs model by letting the disagreement process be driven by an independent L\'evy jump process. Passing from a pure diffusion to a jump-diffusion allows for sudden shifts in market optimism. The formulation covers a number of L\'evy models used in financial engineering, such as Merton's model \cite{Merton1976}, the Variance Gamma (VG) model \cite{Madan1998}, the Normal Inverse Gaussian (NIG) model \cite{BarndorffNielsen1997}, and the double exponential model of Kou \cite{Kou2002}.

The introduction of L\'evy jumps changes the pricing problem. Applying the It\^o--L\'evy formula to the optimal stopping problem leads to a non-local partial integro-differential equation (PIDE) with a moving obstacle. The non-local integral operator, which accounts for the expected change of the bubble value at a jump of the disagreement, destroys the local character of the differential operator. A related problem was studied by Wehowar \cite{Wehowar2018}, who extended the Chen and Kohn framework to a L\'evy setting and observed that classical $C^{2,1}$ solutions need not exist, due to the separation of the continuation and exercise regions by the free boundary.

To handle this we use the theory of viscosity solutions. Viscosity solutions for second-order PDEs were established by Crandall, Ishii, and Lions \cite{Crandall1992}; their extension to non-local PIDEs requires a careful treatment of the integral singularity, as in Alvarez and Tourin \cite{Alvarez1996} and Barles and Imbert \cite{BI08}, and was applied to financial obstacle problems by Cont and Voltchkova \cite{ContVoltchkova2005}.

The main contribution of the paper is the adaptation of the doubling-of-variables argument and the Jensen--Ishii lemma to this non-local obstacle problem. By bounding the jump measure through suitable penalty functions, we prove a comparison principle. We then apply Perron's method, with explicit supersolutions that dominate the non-local L\'evy integral, to prove existence and uniqueness of the bubble price.

The rest of the paper is organized as follows. In Section \ref{sec:bubble-levy} we describe the heterogeneous beliefs model and derive the non-local obstacle PIDE by semimartingale calculus. In Section \ref{sec:viscosity} we introduce the viscosity solution framework and prove its stability under relaxed limits. In Section \ref{sec:comparison} we prove the comparison principle for the non-local operator. In Section \ref{sec:existence} we construct the sub- and supersolution barriers and prove existence and uniqueness of the bubble price by Perron's method. In Section \ref{sec:convergence} we introduce a monotone Implicit--Explicit  finite difference scheme and prove its convergence to the unique viscosity solution, by an extension of the Barles--Souganidis framework. Finally, in Section \ref{sec:simulations} we present numerical simulations for several market scenarios and discuss the results.

\section{The Mathematical Model}\label{sec:bubble-levy}
It is convenient to cast both standard derivative pricing and the heterogeneous beliefs model into a single optimal stopping framework.

Let us define a value function $V(x, t)$ for an asset, where
$t \in [0, T]$ denotes \emph{calendar time} ($t=0$ is the initial trading date,
$t=T$ is maturity). The dynamics of the asset's value rely on an underlying
stochastic state variable $g_t$. Assuming risk-neutral valuation and a constant
continuous discount rate $r > 0$, the general optimal stopping problem is
governed by the following master equation
\begin{equation}\label{main-sde}
    V(g, t) = \sup_{\tau \in [t, T]} \mathbb{E}\!\left[ e^{-r(\tau-t)} \Phi(g_\tau, \tau) \;\Big|\; g_t = g \right],
\end{equation}
where $\tau$ is a stopping time in $[t, T]$, and $\Phi(g, t)$ is the obstacle, that is the immediate payoff obtained when the trade is executed at time $\tau$.


\subsection{Parameters and the Payoff Function}

The payoff obtained from reselling the asset depends on the disagreement variable $g_t$ and on a deterministic time profile $\alpha(t)$. We assume that $\alpha \in C^1([0,T])$, $\alpha(t) \geq 0$ for all $t \in [0,T]$, and $\alpha(T) = 0$. In addition, we assume that $\alpha$ is non-increasing, that is $\alpha'(t) \leq 0$; this is convenient for the construction of the supersolution, and it reflects the fact that the resale premium decays as maturity approaches.

Let $c > 0$ denote the fixed transaction cost of a trade. The net payoff $\psi(g,t)$ at exercise is the linear function
\begin{equation}\label{eq:payoff}
  \psi(g,t) := g\alpha(t) - c.
\end{equation}
This choice provides the complementarity condition needed for the well-posedness of the free-boundary problem. Indeed, the sum of the payoffs of the two groups at any state $g$ is a negative constant
\begin{equation}\label{eq:obstacle-sum}
  \psi(g,t) + \psi(-g,t) = \bigl(g\alpha(t) - c\bigr) + \bigl(-g\alpha(t) - c\bigr) = -2c < 0.
\end{equation}
This condition guarantees that the two groups cannot find it optimal to exercise their resale options at the same time, so that the exercise regions are disjoint and no instantaneous trading loop can occur.

The disagreement process $g_t$ is governed by the mean-reversion rate $\rho \geq 0$, the diffusion volatility $\sigma > 0$, and the L\'evy jump measure $\nu(dz)$. Together with the discount rate $r > 0$, the stability of the model rests on a growth condition rather than on smallness of the parameters. As will be made precise in the viscosity analysis below, the rates $(r, \rho)$ must be large enough to dominate the linear growth of the state-dependent jump variance.
We further assume that the state-dependent jump amplitude $\gamma(g,z)$ is continuous with respect to $g$ and for given $T > 0$ satisfies the following conditions

\begin{equation}\label{eq:param}
\begin{gathered}
    \int_{\mathbb{R}\setminus\{0\}} |\gamma(x,z)-\gamma(y,z)|^2 \,\nu(dz) \leq L_\gamma^2\,|x-y|^2, \quad \forall x,y \in \mathbb{R}, \\[2ex]
    \int_{\mathbb{R}\setminus\{0\}} |\gamma(x,z)|^2 \,\nu(dz) < \infty, \quad \forall x \in \mathbb{R}, \\[2ex]
    \sup_{0<|z|\leq1}\frac{|\gamma(g,z)|}{|z|} \leq K_1\bigl(1+|g|\bigr), \quad \forall g\in\mathbb{R}, \text{ for some } K_1>0.
\end{gathered}
\end{equation}

\subsection{Derivation of the PIDE}

In the heterogeneous beliefs model \cite{HarrisonKreps1978,ScheinkmanXiong2003}, the price premium of the asset (the bubble) comes from the option to resell the asset to the complementary group of investors. The state variable $g_t$ is the difference in optimism between the two groups.

The payoff at exercise (reselling the asset) is not a fixed constant. Following Chen and Kohn \cite{ChenKohn2011} and Berestycki, Monneau, and Scheinkman (BMS) \cite{BMS2014,BMS14}, the seller receives an immediate cash flow net of transaction costs, $\psi(g_\tau, \tau)$, plus the continuation value of the bubble seen from the new buyers, that is at the state $-g_\tau$.

Hence the payoff is coupled to the solution itself,
\[
\Phi_{\text{Bubble}}(g_\tau, \tau) = [V_{\text{Bubble}}(-g_\tau, \tau) + \psi(g_\tau, \tau)]^+.
\]
Applying this to \eqref{main-sde} yields the bubble valuation equation
\[
V_{\text{Bubble}}(g, t) = \sup_{\tau \in [t, T]} \mathbb{E}\!\left[ e^{-r(\tau-t)} \left[ V_{\text{Bubble}}(-g_\tau, \tau) + \psi(g_\tau, \tau) \right]^+ \;\Big|\; g_t = g \right].
\]

\begin{theorem}[Bubble non-local PIDE]\label{thm:FK}
The bubble value defined above
\[
 u(g,t):= \sup_{\tau \in [t, T]} \mathbb{E}\!\left[ e^{-r(\tau-t)} \left[ u(-g_\tau, \tau) + \psi(g_\tau, \tau) \right]^+ \;\Big|\; g_t = g \right]
\]
satisfies the non-local obstacle PIDE
\begin{equation}\label{eq:bubble-pide}
  \min\!\bigl(-\mathcal{L}^\nu u(g,t),\;\;
  u(g,t)-u(-g,t)-\psi(g,t)\bigr) = 0,
  \qquad u(g,T)=0,
\end{equation}
where the forward operator is
\[
  \mathcal{L}^\nu u := u_t + \mathcal{A}u - ru,
\]
\[
  \mathcal{A}u := -\rho g\,u_g + \tfrac{\sigma^2}{2}u_{gg}
  + \int_{\R\setminus\{0\}}\bigl[u(g{+}\gamma)-u(g)-\gamma\,u_g\,\mathbf{1}_{\{|z|\le1\}}
  \bigr]\nu(dz).
\]
and $\gamma \equiv \gamma(g,z)$ represents the state-dependent jump amplitude such that  $\int_{\R\setminus\{0\}}\! |\gamma|^2 \,\nu(dz)<\infty.$
\end{theorem}
\begin{proof}
We give the derivation in detail, following the It\^o--L\'evy
product rule and the resulting variational inequality
(for the general theory of optimal stopping for jump-diffusions, see
\cite[Chapter 3]{OksendalSulem2007} and \cite{ContTankov2004}). The process
$g_s$ is Markov, since both the driving Brownian motion and the L\'evy
process have independent increments. To describe the structure of
these increments, we use the L\'evy--It\^o decomposition
(see \cite[Chapter 2]{Applebaum2009}), by which the
process $g_t$ is written as the sum of a continuous drift, a Brownian
motion, and the jumps,
\begin{equation*}
  dg_t = -\rho\,g_t\,dt + \sigma\,dW_t
  +  \int_{|z| < 1} \gamma(g_{t-}, z) \tilde{N}(dt, dz) + \int_{|z| \ge 1} \gamma(g_{t-}, z) N(dt, dz),
\end{equation*}
\begin{equation*}
  g_0=x.
\end{equation*}

This decomposition lets us treat the continuous and the jump parts separately. In particular, the value function $u(g,t)$ depends only on the current state $g$ and the time $t$.

Fix $t \in [0, T)$. For $s\in[t,T]$, define the discounted value process
\[
Z_s := e^{-r(s-t)}\,u(g_s,\,s), \quad s\in[t,T].
\]
Following the semimartingale calculus of \cite[Chapter 2]{Protter2005},
we apply the It\^o--L\'evy formula to the product $h(s) \cdot F(g_s, s)$ with $h(s)=e^{-r(s-t)}$ and $F(g,s)=u(g,s)$.

For the continuous part of the dynamics, the product rule and the It\^o formula give
\begin{align*}
dZ_s^c &= u(g_s,s)\,d[e^{-r(s-t)}] + e^{-r(s-t)}\,d[u(g_s,s)]^c =\\
&= u(g_s,s)\cdot(-r)\,e^{-r(s-t)}\,ds +\\
& + e^{-r(s-t)}\!\left[ u_t(g_s, s)\,ds + u_g(g_s,s)\,dg_s^c + \frac{1}{2}u_{gg}(g_s,s)\,(dg_s^c)^2 \right],
\end{align*}
where $(dg_s^c)^2 = \sigma^2\,ds$ and $u_t$ denotes the partial derivative with respect to the second (temporal) argument. Here, $g_s^c$ represents the continuous part of the semimartingale $g_s$. Substituting $dg_s^c = -\rho\,g_s\,ds + \sigma\,dW_s$ we get
\begin{align*}
dZ_s^c &= e^{-r(s-t)}\!\left[ u_t - \rho\,g_s\,u_g + \frac{\sigma^2}{2}u_{gg} - ru \right]ds+ \\
& + e^{-r(s-t)}\sigma\,u_g(g_s,s)\,dW_s.
\end{align*}
The $dW_s$ term is a local martingale.

Now we turn to the jump parts. At each jump time $s$, the process $g_s$ jumps by $\Delta g_s = \gamma(g_{s-},z)$, thus
\begin{align*}
\Delta Z_s &= e^{-r(s-t)}\!\bigl[u(g_{s-}+\gamma(g_{s-},z),s) - u(g_{s-},s)\bigr].
\end{align*}
(The discount factor $e^{-r(s-t)}$ is continuous, so it does not jump.) Summing over all jumps up to time $s'$ and introducing the compensated measure $\widetilde{N}(ds,dz) = N(ds,dz) - \nu(dz)ds$ we get
\begin{align*}
\sum_{t<s\leq s'} \Delta Z_s
&= \int_{t}^{s'}\!\int_{\mathbb{R}_0} e^{-r(s-t)}\!\bigl[ u(g_{s-}+\gamma(g_{s-},z),s) - u(g_{s-},s)\bigr] \widetilde{N}(ds,dz) \\
& + \int_{t}^{s'}\!\int_{\mathbb{R}_0} e^{-r(s-t)}\!\bigl[ u(g_{s-}+\gamma(g_{s-},z),s) - u(g_{s-},s)\bigr] \nu(dz)\,ds.
\end{align*}
The $\widetilde{N}$-integral is a local martingale. The $\nu(dz)\,ds$ integral contributes to the predictable drift. To incorporate the truncation function, we add and subtract $\gamma(g_{s-},z)\,u_g(g_{s-},s)\,\mathbf{1}_{|z|\leq 1}$ inside the compensator integral; the subtracted term is absorbed into the drift of the SDE via the $\widetilde{N}$ representation.

Collecting the $ds$-drift from both the continuous and the jump parts, we obtain
\begin{align*}
dZ_s &= e^{-r(s-t)}\!\biggl[
u_t - \rho\,g_s\,u_g + \frac{\sigma^2}{2}u_{gg} - ru \\
& + \int_{\mathbb{R}_0}\!\bigl[u(g_s+\gamma(g_s,z),s)-u(g_s,s)
-\gamma(g_s,z)\,u_g\,\mathbf{1}_{|z|\leq 1}\bigr]\nu(dz)
\biggr]ds \\
& + d\widetilde{M}_s=e^{-r(s-t)}\!\left[ u_t + \mathcal{A} u - ru\right]ds+ d\widetilde{M}_s,
\end{align*}
where $d\widetilde{M}_s$ collects the martingale terms. The spatial terms form the integro-differential generator $\mathcal{A}u$, and we set the forward parabolic operator $\mathcal{L}^\nu u := u_t + \mathcal{A} u - ru$. Thus the dynamics of the discounted process are
\begin{equation}\label{eq:Z-dynamics}
dZ_s = e^{-r(s-t)}\mathcal{L}^\nu u(g_s,s)\,ds + d\widetilde{M}_s.
\end{equation}

In the no-bubble region, $Z_s$ must be a local martingale (the
value of continuing to hold equals the discounted expected future
value).  For this, the $ds$-drift must vanish
\[
  \cL^\nu u(g,t) = u_t + \mathcal{A}u - ru = 0.
\]
This is the PIDE in the continuation (no-bubble) region.

When disagreement $g$ is large enough, the holder of group~$A$'s
asset optimally resells to group~$B$.  Group~$A$ at state $g$
sells to group~$B$ at state $-g$, receiving
\[
  u(g,t) = u(-g,t) + \psi(g,t).
\]
In this region, $Z_s$ is a supermartingale (the holder might have
done better by not exercising yet), so the $ds$-drift is $\leq 0$
which implies 
\[
\cL^\nu u(g,t) =u_t + \mathcal{A} u - ru \le 0,\;\; \text{i.e.}\;\; -u_t - \mathcal{A} u + ru\ge 0.
\]

The two conditions hold on complementary regions, and at any point
at least one holds with equality.  This gives the min-condition
\[
  \min\!\bigl( -u_t - \mathcal{A} u + ru,\; u(g,t)-u(-g,t)-\psi(g,t)\bigr) = 0.
\]
Since $\alpha(T) = 0$, the obstacle satisfies $$\psi(g,T) = g\,\alpha(T) - c = -c<0.$$ As the transaction cost is positive ($c > 0$), the obstacle is negative. The value function of an optimal stopping problem at maturity equals the positive part of the payoff, so $u(g,T) = \max(0, -c) = 0$; that is, the bubble vanishes when the resale opportunity disappears.

We also restrict the viscosity solution to the linear growth class $u(g,t) = \mathcal{O}(|g|)$ as $|g| \to \infty$. This follows from the linear structure of the obstacle $\psi$. By standard optimal stopping theory, the value function inherits the growth class of the reward function, and since the disagreement $g$ enters the obstacle linearly through $g\,\alpha(t)$, the value grows at most linearly in $g$. This $\mathcal{O}(|g|)$ class is used below in the existence proof.
\end{proof}

\section{Viscosity Solutions for the PIDE}\label{sec:viscosity}

Classical solutions require $u\in C^{2,1}$, which need not hold across
the free boundary. We therefore work with viscosity solutions in the sense
of Crandall--Lions, adapted to the non-local obstacle PIDE
\eqref{eq:bubble-pide} following Barles--Imbert~\cite{BI08} and
Cont--Tankov~\cite{ContTankov2004}.

The first argument of the $\min$ in \eqref{eq:bubble-pide} is
$-\mathcal{L}^\nu u = -u_t - \mathcal{A}u + ru$, which
vanishes in the continuation region and is non-negative in the
exercise region.
Since $\mathcal{L}^\nu$ involves $u$ at the distant points
$g+\gamma(g,z)$, the test function $\varphi$ replaces $u$ only in the
local differential terms, while $u$ itself is kept in the
non-local integral and in the obstacle.

\begin{definition}[Viscosity solutions of \eqref{eq:bubble-pide}]
\label{def:visc-pide}
For $(g_0,t_0)\in\mathbb{R}\times(0,T]$ and $\varphi\in C^{2,1}$, define the \emph{test operator}
\begin{multline}\label{eq:testop}
  \mathcal{L}^\nu_\varphi(g_0,t_0;\,u)
  := -\varphi_t(g_0,t_0) + \rho\,g_0\,\varphi_g(g_0,t_0)
  - \tfrac{\sigma^2}{2}\,\varphi_{gg}(g_0,t_0)
  + r\,u(g_0,t_0) \\
  \quad - \int_{\R\setminus\{0\}}\!\Bigl[u(g_0{+}\gamma(g_0,z),t_0)
    - u(g_0,t_0)
    - \gamma(g_0,z)\,\varphi_g(g_0,t_0)\,\mathbf{1}_{|z|\le1}
  \Bigr]\nu(dz).
\end{multline}
Here the local differential terms $(\varphi_t, \varphi_g, \varphi_{gg})$
come from the test function, while the non-local integral and the discount
term $ru$ use the \emph{actual} function $u$.

\noindent\textbf{Viscosity subsolution.} An upper-semicontinuous (USC)
function $u:\mathbb{R}\times[0,T]\to\mathbb{R}$ with $u(\cdot,T)\le 0$ is
a \emph{viscosity subsolution} of \eqref{eq:bubble-pide} if for every
$\varphi\in C^{2,1}$ and every local maximum $(g_0,t_0)$ of $u-\varphi$
with $t_0<T$ the following inequality holds
\begin{equation}\label{eq:visc-sub}
  \min\!\bigl(\mathcal{L}^\nu_\varphi(g_0,t_0;\,u),\;
  u(g_0,t_0)-u(-g_0,t_0)-\psi(g_0,t_0)\bigr)\le 0.
\end{equation}

\noindent\textbf{Viscosity supersolution.} A lower-semicontinuous (LSC)
function $v:\mathbb{R}\times[0,T]\to\mathbb{R}$ with $v(\cdot,T)\ge 0$ is
a \emph{viscosity supersolution} of \eqref{eq:bubble-pide} if for every
$\varphi\in C^{2,1}$ and every local minimum $(g_0,t_0)$ of $v-\varphi$
with $t_0<T$ the following inequality holds
\begin{equation}\label{eq:visc-super}
  \min\!\bigl(\mathcal{L}^\nu_\varphi(g_0,t_0;\,v),\;
  v(g_0,t_0)-v(-g_0,t_0)-\psi(g_0,t_0)\bigr)\ge 0.
\end{equation}
\noindent\textbf{Viscosity solution.}  $u$ is a \emph{viscosity solution}
if $u^*$ (USC regularisation) is a subsolution and $u_*$ (LSC
regularisation) is a supersolution.
\end{definition}

\begin{remark}\label{rem:nonlocal-convention}
In both \eqref{eq:visc-sub} and \eqref{eq:visc-super}, the integral and
the obstacle involve the function $u$ (resp.\ $v$), not
the test function $\varphi$. This is the usual convention for
viscosity solutions of integro-differential equations (see
\cite{BI08,ContTankov2004}): the test function $\varphi$ only
approximates $u$ locally near the contact point $(g_0,t_0)$,
while the non-local terms evaluate $u$ at the distant points
$g_0+\gamma(g_0,z)$ and $-g_0$, where $\varphi$ need not be a good
approximation.
\end{remark}

\begin{remark}\label{rem:testop-pide}
The test operator $\mathcal{L}^\nu_\varphi$ in \eqref{eq:testop}
is the negated forward operator $-\mathcal{L}^\nu$ with the
local terms replaced by those of $\varphi$. For a classical solution
$u\in C^{2,1}$, the choice $\varphi = u$ gives
$\mathcal{L}^\nu_u = -\mathcal{L}^\nu u$, and the PIDE \eqref{eq:bubble-pide} is equivalent to
\eqref{eq:visc-sub}--\eqref{eq:visc-super} with equality throughout.
\end{remark}

\begin{proposition}[Stability]\label{prop:stab-pide}
Let $\{u^\alpha\}_{\alpha\in A}$ be a uniformly locally bounded family of
viscosity subsolutions of \eqref{eq:bubble-pide}. Assume there exists a continuous function $B(g)$ with at most linear growth such that $u^\alpha(g,t) \le B(g)$ globally for all $\alpha$. Then the relaxed upper
semicontinuous envelope
\[
  \bar{u}(g,t) := \limsup_{\substack{(g',t')\to(g,t)\\\alpha\in A}}
  u^\alpha(g',t')
\]
is also a viscosity subsolution of \eqref{eq:bubble-pide}, and
satisfies $\bar{u}(\cdot,T)\le 0$ in the viscosity sense.
\end{proposition}

\begin{proof}
We verify the subsolution condition at a point $(g_0,t_0)$
with $t_0<T$. Let $\varphi\in C^{2,1}$ and suppose $(g_0,t_0)$ is a strict local
maximum of $\bar{u}-\varphi$. Without loss of generality, assume $\bar{u}(g_0,t_0)=\varphi(g_0,t_0)$ and that the maximum is strict in a closed ball $\overline{B_\delta(g_0,t_0)}$ for some $\delta>0$.
By the definition of the relaxed upper limit, there exist sequences
$\alpha_n\in A$ and $(g_n,t_n)\to(g_0,t_0)$ such that
$u^{\alpha_n}(g_n,t_n)\to\bar{u}(g_0,t_0)$.
For each $n$, let $(\hat{g}_n,\hat{t}_n)$ be a maximizer of
$u^{\alpha_n}-\varphi$ in the compact ball $\overline{B_\delta(g_0,t_0)}$. We claim that
$(\hat{g}_n,\hat{t}_n)\to(g_0,t_0)$. Indeed, since
\[
  u^{\alpha_n}(\hat{g}_n,\hat{t}_n)-\varphi(\hat{g}_n,\hat{t}_n)
  \ge u^{\alpha_n}(g_n,t_n)-\varphi(g_n,t_n) \to 0,
\]
any accumulation point $(\hat{g},\hat{t})$ satisfies
$\bar{u}(\hat{g},\hat{t})-\varphi(\hat{g},\hat{t})\ge 0$, and by
the strictness of the maximum in the ball, $(\hat{g},\hat{t})=(g_0,t_0)$. By compactness, the
whole sequence converges.

Since $\nu$ is a Radon measure on $\R\setminus\{0\}$, we evaluate the non-local operator by splitting the jump integral at a radius $r > 0$ whose boundary carries zero measure, $\nu(\{z \in \R\setminus\{0\} \mid |\gamma(g_0, z)| = r\}) = 0$. The interior integral ($|\gamma| \le r$) is evaluated against the test function $\varphi$, and the exterior integral ($|\gamma| > r$) against $u^{\alpha_n}$. The subsolution condition for $u^{\alpha_n}$ gives
\begin{equation}\label{eq:sub-n}
  \min\!\bigl(\mathcal{L}^{\nu, r}_\varphi(\hat{g}_n,\hat{t}_n;\,u^{\alpha_n}),\;
  u^{\alpha_n}(\hat{g}_n,\hat{t}_n)
  -u^{\alpha_n}(-\hat{g}_n,\hat{t}_n)
  -\psi(\hat{g}_n,\hat{t}_n)\bigr)\le 0.
\end{equation}
Since $\varphi\in C^{2,1}$ and
$(\hat{g}_n,\hat{t}_n)\to(g_0,t_0)$, the local differential terms $-\varphi_t$,
$\varphi_g$, $\varphi_{gg}$ at $(\hat{g}_n,\hat{t}_n)$
converge to those at $(g_0,t_0)$. Moreover, since the boundary of the integration domain has $\nu$-measure zero, the indicator function converges almost everywhere. Hence, by dominated convergence on the smooth integrand, the interior jump integral converges,
\begin{multline*}
    \lim_{n\to\infty} \int_{|\gamma|\le r} \bigl[ \varphi(\hat{g}_n+\gamma) - \varphi(\hat{g}_n) - \gamma\varphi_g(\hat{g}_n)\mathbf{1}_{\{|z|\le 1\}}\bigr]\nu(dz) =\\= \int_{|\gamma|\le r} \bigl[ \varphi(g_0+\gamma) - \varphi(g_0) - \gamma\varphi_g(g_0)\mathbf{1}_{\{|z|\le 1\}}\bigr]\nu(dz). 
\end{multline*}
For the exterior integral, the growth bound $B(g)$ is a continuous global upper bound. For each $z\in\mathbb{R}_0$, by the definition of $\bar{u}$,
\[
  \limsup_{n\to\infty}
  u^{\alpha_n}(\hat{g}_n+\gamma(\hat{g}_n,z),\hat{t}_n)
  \le \bar{u}(g_0+\gamma(g_0,z),t_0).
\]
Since the exterior measure $\nu_r(dz) := \mathbf{1}_{|\gamma| > r}\nu(dz)$ is finite and the jump size $\gamma$ grows at most linearly, the bound $B(\hat{g}_n+\gamma)$ is integrable over the exterior domain. It serves as a dominating function for the reverse Fatou lemma applied to the non-negative sequence $B - u^{\alpha_n} \ge 0$. Hence,
\begin{align*}
  \limsup_{n\to\infty}
  \int_{|\gamma|>r}\! u^{\alpha_n}(\hat{g}_n{+}\gamma,\hat{t}_n)\nu(dz)
  &\le
  \int_{|\gamma|>r}\! \bar{u}(g_0{+}\gamma(g_0,z),t_0)\nu(dz).
\end{align*}
Since the  operator $\mathcal{L}^{\nu, r}_\varphi$ subtracts the jump integrals, passing to the limit in the continuous components and using the $\limsup$ bound of the exterior integral, we get
\begin{equation}\label{eq:liminf-integral}
  \liminf_{n\to\infty} \bigl(
  \mathcal{L}^{\nu, r}_\varphi(\hat{g}_n,\hat{t}_n;\,u^{\alpha_n}) \bigr)
  \ge \mathcal{L}^{\nu, r}_\varphi(g_0,t_0;\,\bar{u}).
\end{equation}
By construction, $u^{\alpha_n}(\hat{g}_n,\hat{t}_n)\to\bar{u}(g_0,t_0)$, and since $\bar{u}$ is USC, $\limsup_n u^{\alpha_n}(-\hat{g}_n,\hat{t}_n)\le \bar{u}(-g_0,t_0)$. Thus,
\[
  \liminf_n\bigl[u^{\alpha_n}(\hat{g}_n,\hat{t}_n)-u^{\alpha_n}(-\hat{g}_n,\hat{t}_n)\bigr]
  \ge \bar{u}(g_0,t_0)-\bar{u}(-g_0,t_0).
\]
Taking the $\liminf$ of the minimum condition \eqref{eq:sub-n} and using \eqref{eq:liminf-integral}, we obtain
\[
  \min\!\bigl(\mathcal{L}^{\nu, r}_\varphi(g_0,t_0;\,\bar{u}),\;
  \bar{u}(g_0,t_0)-\bar{u}(-g_0,t_0)-\psi(g_0,t_0)\bigr)\le 0.
\]
Since $r>0$ is an admissible splitting radius, this is the subsolution condition for $\bar{u}$. Finally, as each $u^\alpha(g,T)\le 0$, the relaxed upper limit satisfies 
the terminal 
condition in the viscosity sense, i.e. $\min\bigl(\mathcal{L}^\nu \bar{u}(g,T), \; \bar{u}(g,T)\bigr) \le 0$
\end{proof}

\section{Comparison Principle}\label{sec:comparison}
In this section we prove the comparison principle for viscosity solutions of the PIDE. 
for all $x,y \in \R$. We recall the conditions given in \eqref{eq:param} and assume that the discount rate $r$ and the mean-reversion rate $\rho$ are large enough to dominate the growth of the jump variance, measured by $L_\gamma$. Recall also the obstacle bound $\psi(g,t) + \psi(-g,t) \le -2c.$

\begin{theorem}[Comparison Principle]
	\label{thm:comp-pide}
	Assume \eqref{eq:param}.
	Let $u$ be a subsolution and $v$ a supersolution of \eqref{eq:bubble-pide}
	with at most linear growth. Then $u\leq v$ on $\R\times[0,T)$.
\end{theorem}

\begin{proof}
	We use the doubling-of-variables argument with the Jensen--Ishii
	lemma, treating the state-dependent non-local integral term in full.

	Suppose, for contradiction, that $M:=\sup_{\R\times[0,T)}(u-v)>0$.
	Since $u(g,T)\leq 0 \leq v(g,T)$, the supremum is positive and is not attained at the terminal time $t=T$.

	For parameters $\varepsilon,\beta,\eta>0$, introduce the penalization function $\Phi: \R \times \R \times [0, T) \to \R$,
	\[
	\Phi(x,y,t) := u(x,t) - v(y,t) - \phi(x,y,t),
	\]
	where the penalty function $\phi$ is given by
	\[
	\phi(x, y, t) := \frac{|x-y|^2}{2\varepsilon} + \frac{\beta(x^2+y^2)}{2} + \frac{\eta}{T-t}.
	\]
	
	Since $u$ and $v$ have at most linear growth, $\Phi(x,y,t)\to-\infty$ as $|x|+|y|\to\infty$ (because of the $\beta$ terms) and as $t\to T^-$ (because of the $\eta$ term). Hence $\Phi$ attains its global supremum at a finite point $(\tilde{x},\tilde{y},\tilde{t})$ with $\tilde{t}\in[0,T)$.

	Comparing $\Phi(\tilde{x},\tilde{y},\tilde{t}) \geq \Phi(x^*,x^*,t^*)$ at any approximate maximizer $(x^*,t^*)$ of the difference $u-v$, we control the penalty terms. In particular, the difference of the sub- and supersolution at the doubled variables is bounded below by the supremum minus the error terms,
	\begin{equation}\label{eq:max-props}
		u(\tilde{x},\tilde{t}) - v(\tilde{y},\tilde{t}) \geq M - \mathcal{O}(\varepsilon) - \mathcal{O}(\beta) - \mathcal{O}(\eta).
	\end{equation}

	Since $u$ and $v$ are only semicontinuous, they have no classical derivatives at this maximum, and we use second-order parabolic jets. For an upper semicontinuous function $u$, the second-order superjet at $(x,t)$, denoted $J^{2,+}u(x,t)$, is the set of triples $(a, p, X) \in \R \times \R \times \R$ such that
	\begin{multline*}
	J^{2,+}u(x,t) :=
    \left\{ (a, p, X) \;\middle|\;
     \limsup_{(y,s) \to (x,t)} \frac{u(y,s) - u(x,t) - a(s-t) - p(y-x) - \frac{1}{2}X(y-x)^2}{|s-t| + |y-x|^2} \leq 0 \right\}.
	\end{multline*}
	Similarly, for a lower semicontinuous function $v$, the second-order subjet $J^{2,-}v(x,t)$ is defined by the corresponding lower bound,
	\begin{multline*}
    	J^{2,-}v(x,t) := \left\{ (a, p, X) \;\middle|\; \liminf_{(y,s) \to (x,t)} \frac{v(y,s) - v(x,t) - a(s-t) - p(y-x) - \frac{1}{2}X(y-x)^2}{|s-t| + |y-x|^2} \geq 0 \right\}.
	\end{multline*}
	For compactness when passing to the limit in the penalization parameters, we take the closures of these sets, denoted $\overline{J}^{2,+}u$ and $\overline{J}^{2,-}v$.

	By the maximum principle for semicontinuous functions (the Jensen--Ishii lemma \cite[ Lemma 1/Cor. 2]{BI08}), for any $\varepsilon_0 > 0$ there exist $a_1, a_2 \in \R$ and $X, Y \in \R$, elements of these closed jets
	\[
	(a_1,\, p_x,\, X) \in \overline{J}^{2,+}u(\tilde{x},\tilde{t}) \quad \text{and} \quad (a_2,\, p_y,\, Y) \in \overline{J}^{2,-}v(\tilde{y},\tilde{t}),
	\]
	such that the first-order components equal the derivatives of the penalty function,
	\[
	p_x = \phi_x(\tilde{x}, \tilde{y}, \tilde{t}) = \frac{\tilde{x}-\tilde{y}}{\varepsilon} + \beta\tilde{x},\quad
	p_y = -\phi_y(\tilde{x}, \tilde{y}, \tilde{t}) = \frac{\tilde{x}-\tilde{y}}{\varepsilon} - \beta\tilde{y},
	\]
	\[
	a_1 - a_2 = \phi_t(\tilde{x}, \tilde{y}, \tilde{t}) = \frac{\eta}{(T-\tilde{t})^2}.
	\]
	Moreover, the second-order terms satisfy the block-matrix inequality
	\[
	\begin{pmatrix} X & 0 \\ 0 & -Y \end{pmatrix} \leq D^2\phi + \varepsilon_0 (D^2\phi)^2
	\]
	where $D^2\phi$ is the Hessian of $\phi$ in the spatial variables $(x, y)$. Differentiation gives
	\[ 
	D^2\phi = \begin{pmatrix} \frac{1}{\varepsilon}+\beta & -\frac{1}{\varepsilon} \\ -\frac{1}{\varepsilon} & \frac{1}{\varepsilon}+\beta \end{pmatrix} 
	\]
	Taking $\varepsilon_0 = \varepsilon$, the squared Hessian is
	\[
	(D^2\phi)^2 = \begin{pmatrix} \left(\frac{1}{\varepsilon}+\beta\right)^2 + \frac{1}{\varepsilon^2} & -\frac{2}{\varepsilon}\left(\frac{1}{\varepsilon}+\beta\right) \\ -\frac{2}{\varepsilon}\left(\frac{1}{\varepsilon}+\beta\right) & \left(\frac{1}{\varepsilon}+\beta\right)^2 + \frac{1}{\varepsilon^2} \end{pmatrix}.
	\]
	To decouple $X$ and $Y$, we test the matrix inequality against vectors $\xi \in \R^2$. For $\xi = (1, 1)^T$, the terms of order $\mathcal{O}(1/\varepsilon)$ cancel
	\[
	\xi^T \begin{pmatrix} X & 0 \\ 0 & -Y \end{pmatrix} \xi \leq \xi^T \left[ D^2\phi + \varepsilon (D^2\phi)^2 \right] \xi 
	\]
	\[
	X - Y \leq 2\beta + \varepsilon(2\beta^2) = \mathcal{O}(\beta).
	\]
	This bound controls the difference of the second-order terms independently of $\varepsilon$. To bound $X$ and $-Y$ separately, we test the matrix inequality against the basis vectors. For $e_1 = (1, 0)^T$ we get
	\[
	X \leq \left(\frac{1}{\varepsilon} + \beta\right) + \varepsilon \left[ \left(\frac{1}{\varepsilon} + \beta\right)^2 + \frac{1}{\varepsilon^2} \right].
	\]
	We now substitute $(a_1, p_x, X)$ and $(a_2, p_y, Y)$ into the viscosity inequalities for the two equations to reach a contradiction.
	Since $u$ is a subsolution, we evaluate the operator at $(\tilde{x}, \tilde{t})$, putting the superjet $(a_1, p_x, X) \in \overline{J}^{2,+}u(\tilde{x}, \tilde{t})$ in place of $(\partial_t u, \partial_x u, \partial_{xx} u)$; the subsolution condition reads
	\[ 
	\min \Big( F_u,\; u(\tilde{x}, \tilde{t}) - u(-\tilde{x}, \tilde{t}) - \psi(\tilde{x}, \tilde{t}) \Big) \leq 0, 
	\]
	where $F_u$, the continuous integro-differential operator at the test point, is
	\begin{multline*}
		F_u := -a_1 + \rho \tilde{x} p_x - \frac{\sigma^2}{2}X
        + r u(\tilde{x}, \tilde{t}) - \int_{\R\setminus\{0\}} \Big[ u(\tilde{x} + \gamma(\tilde{x}, z), \tilde{t}) - u(\tilde{x}, \tilde{t}) - \gamma(\tilde{x}, z) p_x \mathbf{1}_{\{|z| \leq 1\}} \Big] \nu(dz).
	\end{multline*}
	This expresses the complementarity of the problem: either $u$ touches the obstacle, or the continuous part $F_u$ is non-positive.

	Similarly, $v$ is a supersolution of the same PIDE. Evaluating the operator at $(\tilde{y}, \tilde{t})$ and substituting the subjet $(a_2, p_y, Y) \in \overline{J}^{2,-}v(\tilde{y}, \tilde{t})$ gives the reverse inequality
	\[ 
	\min \Big( F_v,\; v(\tilde{y}, \tilde{t}) - v(-\tilde{y}, \tilde{t}) - \psi(\tilde{y}, \tilde{t}) \Big) \geq 0,
	\]
	where $F_v$ has the analogous form
	\begin{multline*}
		F_v := -a_2 + \rho \tilde{y} p_y - \frac{\sigma^2}{2}Y + 
        r v(\tilde{y}, \tilde{t})- \int_{\R\setminus\{0\}} \Big[ v(\tilde{y} + \gamma(\tilde{y}, z), \tilde{t}) - v(\tilde{y}, \tilde{t}) - \gamma(\tilde{y}, z) p_y \mathbf{1}_{\{|z| \leq 1\}} \Big] \nu(dz).
	\end{multline*}

	Assume first that at the maximum point $(\tilde{x}, \tilde{y}, \tilde{t})$ of $\Phi$ the obstacle condition holds strictly for both functions, that is $u(\tilde{x}, \tilde{t}) > u(-\tilde{x}, \tilde{t}) + \psi(\tilde{x}, \tilde{t})$ and $v(\tilde{y}, \tilde{t}) > v(-\tilde{y}, \tilde{t}) + \psi(\tilde{y}, \tilde{t})$. Then the equations are active, and by the definitions of viscosity sub- and supersolutions, $F_u \leq 0$ and $F_v \geq 0$.

     Subtracting the second inequality from the first gives $F_u - F_v \leq 0$. Expanding the continuous operators, we obtain the inequality
	\begin{align}
	0 \geq F_u - F_v &= -(a_1 - a_2) + \rho(\tilde{x} p_x - \tilde{y} p_y) - \frac{\sigma^2}{2}(X - Y) \notag \\ 
	&\quad + r\bigl(u(\tilde{x}, \tilde{t}) - v(\tilde{y}, \tilde{t})\bigr)  - \mathcal{I}, \label{eq:master_inequality}
	\end{align}
	where $\mathcal{I}$ is the difference of the non-local integral terms. Write $\gamma_x := \gamma(\tilde{x}, z)$ and $\gamma_y := \gamma(\tilde{y}, z)$. Then
	\begin{equation} \label{eq:integral_def}
	\mathcal{I} := \int_{\R\setminus\{0\}} \biggl[ \Delta u(z) - \Delta v(z) - \bigl(\gamma_x p_x - \gamma_y p_y\bigr)\mathbf{1}_{\{|z| \leq 1\}} \biggr] \nu(dz),
	\end{equation}
	where $\Delta u(z) := u(\tilde{x}+\gamma_x, \tilde{t}) - u(\tilde{x}, \tilde{t})$ and $\Delta v(z) := v(\tilde{y}+\gamma_y, \tilde{t}) - v(\tilde{y}, \tilde{t})$.
	By the form of the penalty, the time-derivative term is $-(a_1 - a_2) = -\frac{\eta}{(T-\tilde{t})^2} < 0.$

	Substituting $p_x = \frac{\tilde{x}-\tilde{y}}{\varepsilon} + \beta\tilde{x}$ and $p_y = \frac{\tilde{x}-\tilde{y}}{\varepsilon} - \beta\tilde{y}$, we get
	\begin{equation*}
	\rho(\tilde{x} p_x - \tilde{y} p_y) = \rho \left[ \tilde{x}\left(\frac{\tilde{x}-\tilde{y}}{\varepsilon} + \beta\tilde{x}\right) - \tilde{y}\left(\frac{\tilde{x}-\tilde{y}}{\varepsilon} - \beta\tilde{y}\right) \right] = \rho \frac{(\tilde{x}-\tilde{y})^2}{\varepsilon} + \rho\beta(\tilde{x}^2 + \tilde{y}^2).
	\end{equation*}
	Since $\rho \geq 0$, this term is non-negative.

	From the block-matrix inequality tested against $(1, 1)^T$, we have $X - Y \leq 2\beta + 2\varepsilon\beta^2$. Hence
	\begin{equation*}
	-\frac{\sigma^2}{2}(X - Y) \geq -\frac{\sigma^2}{2} \bigl( 2\beta + 2\varepsilon\beta^2 \bigr).
	\end{equation*}
	
	By the penalization property \eqref{eq:max-props} and $r > 0$, the discount term is bounded below by
	\[
	r\bigl(u(\tilde{x}, \tilde{t}) - v(\tilde{y}, \tilde{t})\bigr) \geq rM - r\mathcal{O}(\varepsilon) - r\mathcal{O}(\beta) - r\mathcal{O}(\eta).
	\]

	The non-local integral $\mathcal{I}$ requires more care, using the Lipschitz condition \eqref{eq:param}. We use that $\Phi$ is maximal at $(\tilde{x}, \tilde{y}, \tilde{t})$: evaluating $\Phi$ at the jump-shifted points gives $\Phi(\tilde{x}+\gamma_x, \tilde{y}+\gamma_y, \tilde{t}) \leq \Phi(\tilde{x}, \tilde{y}, \tilde{t})$, which yields an upper bound for the uncompensated jump differences,
	\begin{equation*}
	\Delta u(z) - \Delta v(z) \leq \phi(\tilde{x}+\gamma_x, \tilde{y}+\gamma_y, \tilde{t}) - \phi(\tilde{x}, \tilde{y}, \tilde{t}).
	\end{equation*}
	Expanding the explicit form of $\phi$ directly yields
	\begin{multline}\label{eq:jump_bound}
	\Delta u(z) - \Delta v(z)\le \\ \le \frac{1}{2\varepsilon}\bigl[ (\tilde{x}-\tilde{y} + \gamma_x - \gamma_y)^2 - (\tilde{x}-\tilde{y})^2 \bigr] + \frac{\beta}{2}\bigl[ (\tilde{x}+\gamma_x)^2 - \tilde{x}^2 + (\tilde{y}+\gamma_y)^2 - \tilde{y}^2 \bigr] =\\
	= \frac{\tilde{x}-\tilde{y}}{\varepsilon}(\gamma_x - \gamma_y) + \frac{(\gamma_x-\gamma_y)^2}{2\varepsilon} + \beta\tilde{x}\gamma_x + \frac{\beta}{2}\gamma_x^2 + \beta\tilde{y}\gamma_y + \frac{\beta}{2}\gamma_y^2.
	\end{multline}
	Substituting $p_x$ and $p_y$, the first-order compensator term in the L\'evy integral \eqref{eq:integral_def} is
	\[
	\gamma_x p_x - \gamma_y p_y = \frac{\tilde{x}-\tilde{y}}{\varepsilon}(\gamma_x - \gamma_y) + \beta\tilde{x}\gamma_x + \beta\tilde{y}\gamma_y.
	\]
	We insert \eqref{eq:jump_bound} into the integral and split it into small and large jumps.

	\textbf{Small jumps ($|z| \leq 1$).}
	Subtracting the compensator cancels the first-order cross-terms in $\gamma$, leaving only quadratic terms,
	\[
	\mathcal{I}_{small} \leq \int_{|z| \leq 1} \biggl[ \frac{(\gamma_x-\gamma_y)^2}{2\varepsilon} + \frac{\beta}{2}(\gamma_x^2 + \gamma_y^2) \biggr] \nu(dz).
	\]
	By the Lipschitz condition \eqref{eq:param}, this is bounded by
	\[
	\mathcal{I}_{small} \le \frac{L_\gamma^2}{2} \frac{(\tilde{x}-\tilde{y})^2}{\varepsilon} + \frac{\beta}{2} \int_{|z| \leq 1} (\gamma_x^2 + \gamma_y^2) \nu(dz).
	\]

    \textbf{Large jumps ($|z| > 1$).}
	Here there is no compensator, so the first-order cross-terms from the $\phi$ expansion remain. We split the integral into a part scaled by $1/\varepsilon$ and a part scaled by $\beta.$ We get
	\begin{multline*}
	\mathcal{I}_{large} \le \int_{|z|>1} \biggl[ \frac{\tilde{x}-\tilde{y}}{\varepsilon}(\gamma_x-\gamma_y) + \frac{(\gamma_x-\gamma_y)^2}{2\varepsilon} \biggr] \nu(dz) + \int_{|z|>1} \biggl[ \frac{\beta}{2}(\gamma_x^2 + \gamma_y^2) + \beta(\tilde{x}\gamma_x + \tilde{y}\gamma_y) \biggr] \nu(dz).
	\end{multline*}
	
	For the $1/\varepsilon$ part, since the total mass of large jumps $\lambda := \int_{|z|>1} 1 \, \nu(dz)$ is finite, we bound the linear term by the Cauchy--Schwarz inequality with respect to $\nu$ as follows
	\[
	\frac{|\tilde{x}-\tilde{y}|}{\varepsilon} \int_{|z|>1} |\gamma_x-\gamma_y| \nu(dz) \le \frac{|\tilde{x}-\tilde{y}|}{\varepsilon} \sqrt{\int_{\R\setminus\{0\}}|\gamma_x-\gamma_y|^2 \nu(dz)} \sqrt{\lambda} \le L_\gamma\sqrt{\lambda}\frac{(\tilde{x}-\tilde{y})^2}{\varepsilon}.
	\]
	The quadratic term in the $1/\varepsilon$ part is bounded by integrating the Lipschitz assumption \eqref{eq:param},
	\[
	\int_{|z|>1} \frac{(\gamma_x-\gamma_y)^2}{2\varepsilon} \nu(dz) \le \frac{L_\gamma^2}{2} \frac{(\tilde{x}-\tilde{y})^2}{\varepsilon}.
	\]

	For the $\beta$ part, the jump amplitude grows at most linearly, so $\int_{|z|>1} \gamma_x^2 \nu(dz) \le C(1 + \tilde{x}^2)$ for some constant $C$. By Cauchy--Schwarz, the cross-term $\beta \tilde{x} \gamma_x$ is bounded by a term proportional to $\beta \tilde{x}^2$. In the doubling-of-variables method, the penalty terms $\beta \tilde{x}^2$ and $\beta \tilde{y}^2$ are uniformly bounded and tend to zero as $\beta \to 0$, so the $\beta$ integral is absorbed into an error term $\mathcal{O}(\beta)$.

	Adding the bounds for $\mathcal{I}_{small}$ and $\mathcal{I}_{large}$ we obtain
	\[
	\mathcal{I} = \mathcal{I}_{small} + \mathcal{I}_{large} \le \biggl(L_\gamma^2 + L_\gamma\sqrt{\lambda}\biggr) \frac{(\tilde{x}-\tilde{y})^2}{\varepsilon} + \mathcal{O}(\beta).
	\]
	Hence $-\mathcal{I} \ge -C_1 \frac{(\tilde{x}-\tilde{y})^2}{\varepsilon} - \mathcal{O}(\beta)$, where $C_1 := L_\gamma^2 + L_\gamma\sqrt{\lambda}$ depends only on the jump parameters.

	We now pass to the limit in \eqref{eq:master_inequality}, in the order $\varepsilon \to 0$, then $\beta \to 0$, then $\eta \to 0$.
	By the standard properties of the penalization, the points $(\tilde{x}_\varepsilon, \tilde{y}_\varepsilon)$ converge to some $(x_\beta, x_\beta)$ as $\varepsilon \to 0$, and the penalty $\frac{|\tilde{x}_\varepsilon - \tilde{y}_\varepsilon|^2}{\varepsilon} \to 0$. In particular, the error term $C_1 \frac{(\tilde{x}-\tilde{y})^2}{\varepsilon}$ vanishes. The difference $u(\tilde{x}_\varepsilon, \tilde{t}_\varepsilon) - v(\tilde{y}_\varepsilon, \tilde{t}_\varepsilon)$ converges to a value $M_{\beta, \eta}$, and $2\varepsilon\beta^2 \to 0$.

	Next, as $\beta \to 0$, the maximum point converges to $(x_\eta, t_\eta)$, the $\beta$-terms (the diffusion bound $2\beta$ and the integral bound $\mathcal{O}(\beta)$) vanish, the value becomes $M_\eta$, and the inequality reduces to
	\[
	0 \geq -\frac{\eta}{(T-t_\eta)^2} + r M_\eta.
	\]
Rearranging this inequality yields $r M_\eta \le \frac{\eta}{(T-t_\eta)^2}$. 
To evaluate the limit as $\eta \to 0$, we rely on the terminal conditions. Because the barriers guarantee $u(g,T) = v(g,T) = 0$, and we assumed the global supremum $M := \sup_{\R \times [0,T)} (u-v)$ is strictly positive ($M > 0$), the true unpenalized supremum must be attained at some strictly interior time $t^* < T$. 

As $\eta \to 0$, the penalized maximum $(x_\eta, t_\eta)$ converges to this true maximum, meaning $t_\eta \to t^*$. Consequently, the denominator $(T-t_\eta)^2$ is strictly bounded away from zero (converging to $(T-t^*)^2 > 0$). As the temporal parameter is removed, the supremum $M_\eta \to M$, and the bounding penalty term $\frac{\eta}{(T-t_\eta)^2} \to 0$. This  implies $r M \le 0.$
However, by assumption $M > 0$ and $r > 0$, we have $r M > 0$, which strictly contradicts $0 \geq r M$. Therefore, we must have $M \le 0$, which proves $u \le v$ globally.

	We now treat the remaining case. Suppose that along a subsequence (still indexed by $\varepsilon$) the equation is not active for $u$ at the maximum point, but the obstacle is, that is
	\begin{equation}\label{eq:obstacle_u}
	u(\tilde{x}, \tilde{t}) - u(-\tilde{x}, \tilde{t}) \leq \psi(\tilde{x}, \tilde{t}).
	\end{equation}
	By the definition of a supersolution, $v$ lies above the obstacle everywhere, so at $\tilde{y}$,
	\begin{equation}\label{eq:obstacle_v}
	v(\tilde{y}, \tilde{t}) - v(-\tilde{y}, \tilde{t}) \geq \psi(\tilde{y}, \tilde{t}).
	\end{equation}
	Subtracting \eqref{eq:obstacle_v} from \eqref{eq:obstacle_u} gives an upper bound for the difference of the functions at the reflected points,
	\begin{equation}\label{eq:obstacle_diff_upper}
	\bigl(u(\tilde{x}, \tilde{t}) - v(\tilde{y}, \tilde{t})\bigr) - \bigl(u(-\tilde{x}, \tilde{t}) - v(-\tilde{y}, \tilde{t})\bigr) \leq \psi(\tilde{x}, \tilde{t}) - \psi(\tilde{y}, \tilde{t}).
	\end{equation}
	We obtain a lower bound for this difference from the symmetry of $\Phi$. Recall that
	\[
	\phi(x, y, t) = \frac{(x-y)^2}{2\varepsilon} + \frac{\beta(x^2 + y^2)}{2} + \frac{\eta}{T-t},
	\]
	so $\phi$ is invariant under reflection, $\phi(x, y, t) = \phi(-x, -y, t)$. Since $(\tilde{x}, \tilde{y}, \tilde{t})$ is the maximum of $\Phi$, we have $\Phi(\tilde{x}, \tilde{y}, \tilde{t}) \geq \Phi(-\tilde{x}, -\tilde{y}, \tilde{t})$; expanding and using the symmetry of $\phi$, the penalty terms cancel and
	\begin{equation}\label{eq:obstacle_diff_lower}
	u(\tilde{x}, \tilde{t}) - v(\tilde{y}, \tilde{t}) \geq u(-\tilde{x}, \tilde{t}) - v(-\tilde{y}, \tilde{t}).
	\end{equation}
	Combining \eqref{eq:obstacle_diff_lower} and \eqref{eq:obstacle_diff_upper}, the difference of the functions is trapped between zero and the variation of the payoff,
	\begin{equation}\label{eq:squeeze}
	0 \leq \bigl(u(\tilde{x}, \tilde{t}) - v(\tilde{y}, \tilde{t})\bigr) - \bigl(u(-\tilde{x}, \tilde{t}) - v(-\tilde{y}, \tilde{t})\bigr) \leq \psi(\tilde{x}, \tilde{t}) - \psi(\tilde{y}, \tilde{t}).
	\end{equation}
	Let $\varepsilon \to 0$. Up to a subsequence, the maximizers converge, $(\tilde{x}, \tilde{y}, \tilde{t}) \to (\hat{x}, \hat{x}, \hat{t})$. Since $|\tilde{x} - \tilde{y}| \to 0$ and $\psi$ is uniformly continuous, the right-hand side of \eqref{eq:squeeze} tends to zero, and the squeeze gives
	\begin{equation}\label{eq:reflected_max}
	u(\hat{x}, \hat{t}) - v(\hat{x}, \hat{t}) = u(-\hat{x}, \hat{t}) - v(-\hat{x}, \hat{t}).
	\end{equation}
	This equality shows that the reflected point $(-\hat{x}, \hat{t})$ also attains the global supremum of the limiting difference $u(x,t) - v(x,t) - \beta x^2 - \frac{\eta}{T-t}$.

	Moreover, by the upper semicontinuity of $u$ and the continuity of $\psi$, the obstacle condition \eqref{eq:obstacle_u} passes to the limit,
	\begin{equation}\label{eq:obstacle_limit_1}
	u(\hat{x}, \hat{t}) - u(-\hat{x}, \hat{t}) \leq \psi(\hat{x}, \hat{t}).
	\end{equation}
	We now derive a contradiction. Since $(-\hat{x}, \hat{t})$ is also a global maximum, we may repeat the doubling argument centered at this reflected point. Suppose the obstacle condition for $u$ is also active there, that is
	\begin{equation}\label{eq:obstacle_limit_2}
	u(-\hat{x}, \hat{t}) - u(\hat{x}, \hat{t}) \leq \psi(-\hat{x}, \hat{t}).
	\end{equation}
	Adding \eqref{eq:obstacle_limit_1} and \eqref{eq:obstacle_limit_2} eliminates $u$ and gives
	\[
	0 \leq \psi(\hat{x}, \hat{t}) + \psi(-\hat{x}, \hat{t}).
	\]
	But by the positivity of the transaction cost in \eqref{eq:param}, the sum of the obstacles satisfies $\psi(x,t) + \psi(-x,t) \le -2c < 0$, so $0 \leq -2c < 0$, a contradiction.

	Hence the obstacle condition cannot be active for $u$ at $(-\hat{x}, \hat{t})$, so the equation must be active there. Centering the argument at $(-\hat{x}, \hat{t})$, we are back in the previous case, which gives the contradiction $rM \leq 0$.

	This covers all cases and completes the proof.
	
\end{proof}

\begin{corollary}[Uniqueness]\label{cor:unique}
	The viscosity solution of \eqref{eq:bubble-pide} is unique among
	functions with at most linear growth.
\end{corollary}

\begin{proof}
	If $u$ and $v$ are both viscosity solutions, then $u$ is a subsolution
	and $v$ is a supersolution, so $u\leq v$ by the comparison principle.
	Reversing the roles ($v$ is a subsolution, $u$ is a supersolution)
	gives $v\leq u$. Hence $u=v$.
\end{proof}


\section{Existence of Viscosity Solutions}\label{sec:existence}

To prove existence of a unique viscosity solution, we use Perron's method. For this we construct sub- and supersolutions that serve as lower and upper barriers for the bubble price.

\subsection{The Subsolution Barrier}

\begin{lemma}\label{lem:sub-pide}
The intrinsic payoff function, defined as $\underline{u}(g,t) := \max(0, \psi(g,t))$, is a viscosity subsolution of \eqref{eq:bubble-pide}.
\end{lemma}

\begin{proof}
We must check that for every point $(g_0,t_0)$ with $t_0 < T$ and every test function $\varphi \in C^{2,1}$ such that $\underline{u}-\varphi$ has a local maximum at $(g_0,t_0)$,
\[
  \min\!\bigl(\cL^\nu_\varphi(g_0,t_0; \underline{u}),\; \underline{u}(g_0,t_0) - \underline{u}(-g_0,t_0) - \psi(g_0,t_0)\bigr) \leq 0.
\]
We consider three regions, according to the sign of $\psi(g_0,t_0)$.

\textbf{Region 1: $\psi(g_0,t_0) < 0$.} \\
Here $\underline{u}(g,t) = 0$ in a neighborhood of $(g_0,t_0)$. Since $\underline{u}-\varphi$ has a local maximum at $(g_0,t_0)$ and $\underline{u}$ is locally constant, $\varphi$ has a local minimum at $(g_0,t_0)$, so $\varphi_t = 0$, $\varphi_g = 0$, and $\varphi_{gg} \geq 0$.

The local part of the test operator is then
\[
  -\varphi_t + \rho g_0\,\varphi_g - \tfrac{\sigma^2}{2}\varphi_{gg} + r\,\underline{u}(g_0,t_0)
  \;=\; 0 + 0 - \tfrac{\sigma^2}{2}\varphi_{gg} + 0 \;\leq\; 0.
\]
For the non-local term, $\underline{u}(g_0+\gamma(g_0,z),t_0) - \underline{u}(g_0,t_0) = \max(0, \psi(g_0+\gamma(g_0,z),t_0)) \geq 0$, and since $\varphi_g = 0$ the truncation term vanishes, so the integral is
\[
  -\int_{\R\setminus\{0\}}\!\bigl[\underline{u}(g_0+\gamma,t_0) - \underline{u}(g_0,t_0)\bigr]\nu(dz) 
  \;=\; -\int_{\R\setminus\{0\}}\max(0,\psi(g_0+\gamma,t_0))\,\nu(dz) \;\leq\; 0.
\]
Adding these, $\cL^\nu_\varphi(g_0,t_0) \leq 0$, so the first argument of the minimum is non-positive and the subsolution condition holds.

\textbf{Region 2: $\psi(g_0,t_0) > 0$.} \\
Here the barrier equals the payoff, $\underline{u}(g_0,t_0) = \psi(g_0,t_0) > 0$. By \eqref{eq:param}, $\psi(g_0,t_0) + \psi(-g_0,t_0) \leq -2c$, and since $\psi(g_0,t_0) > 0$ and $c > 0$,
\[
  \psi(-g_0,t_0) \leq -2c - \psi(g_0,t_0) < -2c < 0.
\]
Thus at the reflected point the payoff is negative, so
\[
  \underline{u}(-g_0,t_0) = \max(0, \psi(-g_0,t_0)) = 0,
\]
and the second argument of the minimum is
\begin{align*}
  \underline{u}(g_0,t_0) - \underline{u}(-g_0,t_0) - \psi(g_0,t_0)
  &= \psi(g_0,t_0) - 0 - \psi(g_0,t_0) = 0 \;\leq\; 0.
\end{align*}
which is zero, so the subsolution condition holds through the obstacle.

\textbf{Region 3: $\psi(g_0,t_0) = 0$.} \\
Here $\underline{u}(g_0,t_0) = 0$. As above, $\psi(-g_0,t_0) \leq -2c < 0$, so $\underline{u}(-g_0,t_0) = 0$, and $\underline{u}(g_0,t_0) - \underline{u}(-g_0,t_0) = 0 \leq \psi(g_0,t_0) = 0$.

\textbf{Terminal condition.}\\
By the terminal condition of Theorem \ref{thm:FK}, the resale option is worthless at maturity, $\psi(g,T) \leq 0$, so
\[
  \underline{u}(g,T) = \max(0, \psi(g,T)) = 0,
\]
which is the required terminal condition for a subsolution.
\end{proof}

\subsection{The Supersolution Barrier}

To stay in the linear growth uniqueness class of Corollary \ref{cor:unique}, the supersolution must grow at most linearly as $|g| \to \infty$. Let $\nu(dz)$ be the L\'evy measure, and let $\lambda := \nu(\{z \in \R\setminus\{0\} \mid |z| > 1\}) < \infty$ be the intensity of large jumps.

\begin{lemma}\label{lem:super-pide}
Define the smooth function $f(g) := k\sqrt{g^2+M} + \frac{g}{2} + K$ for a  parameter $k > 1/2$. Assume the variance bound in \eqref{eq:param} holds, the time profile is non-increasing ($\alpha'(t) \le 0$), and the continuous drift dominates the jump growth, that is
 \begin{equation}\label{eq:growth_cond}
r + \rho > \sqrt{\lambda} L_\gamma + \frac{1}{2} L_\gamma^2.
\end{equation}
Then for any smoothing constant $M > 0$ there exist a constant $k=k^*>1/2$ and a large enough constant $K > 0$ such that $\tilde{u}(g,t) := \alpha(t)f(g)$ is a global supersolution of the PIDE \eqref{eq:bubble-pide}.
\end{lemma}

\begin{proof}
Let $h(g) := k\sqrt{g^2+M} + \frac{g}{2}$, so that $f(g) = h(g) + K$. We check the obstacle and the differential conditions separately. For $k > 1/2$, $h(g)$ is positive on $\R$, so $f(g) \ge K > 0$, and $f'(g) = h'(g) = k\frac{g}{\sqrt{g^2+M}} + \frac{1}{2}$.

We require $\tilde{u}(g,t) - \tilde{u}(-g,t) \geq \psi(g,t) = g\alpha(t) - c$. Since $K$ and $k\sqrt{g^2+M}$ cancel in the difference,
$$
  \tilde{u}(g,t) - \tilde{u}(-g,t) = \alpha(t)\biggl[ \Bigl(h(g) + K\Bigr) - \Bigl(h(-g) + K\Bigr) \biggr] = \alpha(t)g.
$$
The condition becomes $\alpha(t)g \ge \alpha(t)g - c$, that is $0 \ge -c$. As $c > 0$, the obstacle condition holds for every $k$.

Let $\mathcal{P} u := u_t - \rho g u_g + \frac{1}{2}\sigma^2 u_{gg} - r u$ be the continuous parabolic operator, and write the negated PIDE as $-\mathcal{L}^\nu \tilde{u} = -\mathcal{P} \tilde{u} - J$.
Using $\tilde{u}_t = \alpha'(t)f(g)$, $\tilde{u}_g = \alpha(t)h'(g)$, and $\tilde{u}_{gg} = \alpha(t)h''(g)$, the negated parabolic operator is
$$
  -\mathcal{P}\tilde u = \alpha(t) \biggl[ \Bigl(r - \frac{\alpha'(t)}{\alpha(t)}\Bigr)f(g) + \rho g h'(g) - \frac{\sigma^2 k M}{2(g^2+M)^{3/2}} \biggr].
$$
For the term $\rho g h'(g)$ we have
\begin{multline*}
  \rho g h'(g) = \rho\biggl( \frac{k g^2}{\sqrt{g^2+M}} + \frac{g}{2} \biggr) = \rho\biggl( \frac{k(g^2+M)}{\sqrt{g^2+M}} - \frac{k M}{\sqrt{g^2+M}} + \frac{g}{2} \biggr) = \rho h(g) - \frac{\rho k M}{\sqrt{g^2+M}}.
\end{multline*}
Substituting this and separating the growth part $h(g)$ from the constant $K$, we get
$$
  \Bigl(r - \frac{\alpha'(t)}{\alpha(t)}\Bigr)\bigl(h(g) + K\bigr) + \rho h(g) = \Bigl(r + \rho - \frac{\alpha'(t)}{\alpha(t)}\Bigr)h(g) + \Bigl(r - \frac{\alpha'(t)}{\alpha(t)}\Bigr)K.
$$
The continuous operator then becomes
$$
  -\mathcal{P}\tilde u = \alpha(t) \biggl[ \Bigl(r + \rho - \frac{\alpha'(t)}{\alpha(t)}\Bigr)h(g) + \Bigl(r - \frac{\alpha'(t)}{\alpha(t)}\Bigr)K - \frac{\rho k M}{\sqrt{g^2+M}} - \frac{\sigma^2 k M}{2(g^2+M)^{3/2}} \biggr].
$$

The L\'evy integral is $J = \alpha(t) \int_{\R\setminus\{0\}} [ f(g+\gamma) - f(g) - \gamma f'(g) \mathbf{1}_{\{|z|\le 1\}} ] \nu(dz)$. The constant $K$ cancels in the jump difference $f(g+\gamma)-f(g) = h(g+\gamma)-h(g)$. We split the integral into large and small jumps.

For large jumps ($|z| > 1$), the uncompensated difference is bounded by $\|h'\|_\infty |\gamma(g,z)| = (k + 1/2)|\gamma(g,z)|$. By Cauchy--Schwarz with the finite exterior measure $\lambda$,
\begin{multline*}
J_{large} \le \alpha(t)\left(k + \frac{1}{2}\right) \int_{|z|>1} |\gamma(g, z)| \nu(dz) \le  \alpha(t) \left(k + \frac{1}{2}\right)\sqrt{\lambda} \sqrt{\int_{|z|>1} |\gamma(g, z)|^2 \nu(dz)}.
\end{multline*}
By \eqref{eq:param}, $\int_{\R\setminus\{0\}} |\gamma(x, z) - \gamma(y, z)|^2 \nu(dz) \le L_\gamma^2 |x-y|^2$. To separate the state-dependent growth from the baseline jump variance, we use Minkowski's inequality in $L^2(\nu)$. With $C_{large} := \int_{|z|>1} \gamma(0,z)^2 \nu(dz) < \infty$,
\begin{multline*}
\sqrt{\int_{|z|>1} |\gamma(g, z)|^2 \nu(dz)} \le \sqrt{\int_{|z|>1} |\gamma(g, z) - \gamma(0, z)|^2 \nu(dz)} + \sqrt{C_{large}} \le L_\gamma |g| + \sqrt{C_{large}}.
\end{multline*}
Substituting this bound gives the large-jump estimate
$$
J_{large}\le \alpha(t) \left(k + \frac{1}{2}\right) \sqrt{\lambda} \Bigl( L_\gamma |g| + \sqrt{C_{large}} \Bigr).
$$
For the small-jump integral ($|z| \le 1$), we compute the uncompensated second difference of the test function. The linear term $g/2$ cancels, and the remainder $R$ is
\begin{multline*}
    R := h(g+\gamma)-h(g)-\gamma h'(g) =\\= \left( k\sqrt{(g+\gamma)^2+M} + \frac{g}{2} + \frac{\gamma}{2}  \right) - \left( k\sqrt{g^2+M} + \frac{g}{2}\right) - \left( k\gamma \frac{g}{\sqrt{g^2+M}} + \frac{\gamma}{2} \right)=\\= k\biggl( \sqrt{(g+\gamma)^2+M} - \sqrt{g^2+M} - \gamma\frac{g}{\sqrt{g^2+M}} \biggr).
\end{multline*}
  
Set $A := \sqrt{g^2+M}$ and $B := \sqrt{(g+\gamma)^2+M}$, both positive. Then
$$
  R =k\left( B - A - \frac{\gamma g}{A}\right).
$$
Over the common denominator $A$,
$$
  R = k\cdot\frac{AB - A^2 - \gamma g}{A}.
$$
Using $A^2 = g^2 + M$ in the numerator,
\begin{align*}
  R &= k\cdot\frac{AB - (g^2 + M) - \gamma g}{A} \\
    &= k\cdot\frac{AB - M - g(g + \gamma)}{A}.
\end{align*}
To remove the mixed radical $AB$, multiply numerator and denominator by the positive conjugate $AB + M + g(g+\gamma)$. By the difference of squares,
\begin{equation}\label{eq:R_conjugate}
  R = k\cdot\frac{(AB)^2 - \bigl(M + g(g+\gamma)\bigr)^2}{A\bigl(AB + M + g(g+\gamma)\bigr)}.
\end{equation}
We expand the numerator. For the first term,
\begin{align*}
  (AB)^2 &= g^2(g+\gamma)^2 + M(g+\gamma)^2 + M g^2 + M^2,
\end{align*}
and for the second,
$$
  \bigl(M + g(g+\gamma)\bigr)^2 = M^2 + 2Mg(g+\gamma) + g^2(g+\gamma)^2.
$$
Subtracting, the terms $g^2(g+\gamma)^2$ and $M^2$ cancel,
\begin{multline*}
  \Bigl[ g^2(g+\gamma)^2 + M(g+\gamma)^2 + M g^2 + M^2 \Bigr] - \Bigl[ M^2 + 2Mg(g+\gamma) + g^2(g+\gamma)^2 \Bigr] =\\
  = M(g+\gamma)^2 + Mg^2 - 2Mg(g+\gamma) = M\gamma^2.
\end{multline*}
so \eqref{eq:R_conjugate} becomes the identity
\begin{equation*}
  R = k\cdot\frac{M\gamma^2}{A\bigl(AB + M + g(g+\gamma)\bigr)}.
\end{equation*}
By $g^2+(g+\gamma)^2 \ge 2|g(g+\gamma)|$, we have $A^2 B^2 = g^2(g+\gamma)^2 + M(g^2+(g+\gamma)^2) + M^2 \ge (|g(g+\gamma)| + M)^2$, hence $AB \ge |g(g+\gamma)| + M$. The term in the denominator is then bounded below by $|g(g+\gamma)| + g(g+\gamma) + 2M \ge 2M$, which gives $R \le \frac{k\gamma^2}{2\sqrt{g^2+M}}$.

Integrating this bound against the L\'evy measure, the small-jump integral satisfies
\begin{equation}\label{eq:Jsmall_integral}
  J_{small} \le \frac{k\alpha(t)}{2\sqrt{g^2+M}} \int_{|z|\le 1} \gamma(g,z)^2 \nu(dz).
\end{equation}

For the integral of the squared jump amplitude we again use Minkowski's inequality in $L^2$ over $\{|z| \le 1\}$. With $C_{small} := \int_{|z|\le 1} \gamma(0,z)^2 \nu(dz)$,
\begin{multline*}
  \sqrt{\int_{|z|\le 1} \gamma(g,z)^2 \nu(dz)} \le  \sqrt{\int_{|z|\le 1} |\gamma(g,z) - \gamma(0,z)|^2 \nu(dz)} + \sqrt{C_{small}} \le L_\gamma |g| + \sqrt{C_{small}}.
\end{multline*}
Squaring gives
$$
  \int_{|z|\le 1} \gamma(g,z)^2 \nu(dz) \le L_\gamma^2 g^2 + 2L_\gamma \sqrt{C_{small}} |g| + C_{small},
$$
and substituting into \eqref{eq:Jsmall_integral},
$$
  J_{small} \le \alpha(t) \frac{k\bigl(L_\gamma^2 g^2 + 2L_\gamma \sqrt{C_{small}} |g| + C_{small}\bigr)}{2\sqrt{g^2+M}}.
$$
As $|g| \to \infty$, this behaves like $\frac{k}{2} L_\gamma^2 |g|$.

Combining the bounds, the requirement $-\mathcal{L}^\nu \tilde{u} = -\mathcal{P} \tilde{u} - J_{small}-J_{large}\ge 0$ becomes
\begin{multline*}
  \Bigl(r + \rho - \frac{\alpha'(t)}{\alpha(t)}\Bigr)h(g) + \Bigl(r - \frac{\alpha'(t)}{\alpha(t)}\Bigr)K \ge  \left(k + \frac{1}{2}\right)\sqrt{\lambda} \Bigl( L_\gamma |g| + \sqrt{C_{large}} \Bigr) +\\+ \frac{k\bigl(L_\gamma^2 g^2 + 2L_\gamma \sqrt{C_{small}} |g|\bigr)}{2\sqrt{g^2+M}} 
  + \frac{\rho k M}{\sqrt{g^2+M}} + \frac{\sigma^2 k M}{2(g^2+M)^{3/2}} + \frac{k C_{small}}{2\sqrt{g^2+M}}.
\end{multline*}
By assumption $\alpha$ is non-increasing ($\alpha'(t) \le 0$), so $-\alpha'(t)/\alpha(t) \ge 0$. Since $h(g)$ and $K$ are positive, we bound the left-hand side below by dropping the non-negative time-derivative terms,
$$
\Bigl(r + \rho - \frac{\alpha'(t)}{\alpha(t)}\Bigr)h(g) + \Bigl(r - \frac{\alpha'(t)}{\alpha(t)}\Bigr)K \ge (r + \rho)h(g) + rK.
$$
The right-hand side is now independent of time.
Let
$$
D(g; M, k) := \text{RHS}(g) - (r+\rho)h(g)
$$
be the spatial deficit. The supersolution condition $-\mathcal{L}^\nu \tilde{u} \ge 0$ holds provided $rK \ge D(g; M, k)$ for all $g \in \R$.
We evaluate the deficit as $g \to \pm\infty$.

The fraction coming from the small jumps grows linearly,
\[
  \lim_{|g| \to \infty} \frac{k\bigl(L_\gamma^2 g^2 + 2L_\gamma \sqrt{C_{small}} |g|\bigr)}{2|g|\sqrt{g^2+M}} = \lim_{|g| \to \infty} \frac{kL_\gamma^2|g| +2kL_\gamma \sqrt{C_{small}}  }{2|g|\sqrt{1+M/g^2}} =k\cdot\frac{1}{2} L_\gamma^2.
\]
Dividing $D(g; M, k)$ by $|g|$ and letting $|g| \to \infty$, the fractional penalty terms (mean-reversion, diffusion, and the small-jump origin variance) decay like $\mathcal{O}(|g|^{-2})$ or $\mathcal{O}(|g|^{-4})$ and vanish. The remaining terms — the large-jump compensator penalty, the small-jump variance, and the drift $(r+\rho)h(g)$ — grow linearly, with a leading coefficient depending on $k > 1/2$. Since $h(g) = k\sqrt{g^2+M} + \frac{g}{2}$ grows as $(k+1/2)g$ for $g>0$ and as $(k-1/2)|g|$ for $g<0$, this coefficient differs in the two directions, and we treat them separately.

\textbf{Case 1: $g \to +\infty$.}
Here $h(g) \sim (k + 1/2)g$. Dividing $D(g; M, k)$ by $g$ and taking the limit, the fractional terms vanish and the linear terms dominate,
$$
  \limsup_{g \to +\infty} \frac{D(g; M, k)}{|g|} \le \biggl(k + \frac{1}{2}\biggr)\sqrt{\lambda} L_\gamma + \frac{k}{2} L_\gamma^2 - (r+\rho)\biggl(k + \frac{1}{2}\biggr).
$$
For $D(g; M, k) \to -\infty$, this coefficient must be negative, which gives
$$
  r + \rho > \sqrt{\lambda} L_\gamma + \frac{k}{k + 1/2}\biggl(\frac{1}{2} L_\gamma^2\biggr).
$$

\textbf{Case 2: $g \to -\infty$.}
Here $h(g) \sim (k - 1/2)|g|$, so the drift penalizes the system in proportion to this slope. Dividing by $|g|$ and taking the limit,
$$
  \limsup_{g \to -\infty} \frac{D(g; M, k)}{|g|} \le \biggl(k + \frac{1}{2}\biggr)\sqrt{\lambda} L_\gamma + \frac{k}{2} L_\gamma^2 - (r+\rho)\biggl(k - \frac{1}{2}\biggr).
$$
Requiring this coefficient negative gives the condition for the left tail,
$$
  r + \rho > \frac{k + 1/2}{k - 1/2}\sqrt{\lambda} L_\gamma + \frac{k}{k - 1/2}\biggl(\frac{1}{2} L_\gamma^2\biggr).
$$
We must satisfy both conditions. Since $k > 1/2$, we have $k - 1/2 < k + 1/2$, hence
$$
\frac{k}{k - 1/2} > \frac{k}{k + 1/2}\quad \text{and}\quad\frac{k + 1/2}{k - 1/2}> 1,
$$
so Case~2 implies Case~1. By assumption \eqref{eq:growth_cond}, $r + \rho > \sqrt{\lambda} L_\gamma + \frac{1}{2} L_\gamma^2$, so there is a gap $\epsilon > 0$. Since $\lim_{k \to \infty} \frac{k}{k - 1/2} =\lim_{k \to \infty} \frac{k+1/2}{k - 1/2}= 1$, there is a finite $k^* > 1/2$ large enough that
$$
  r + \rho > \frac{k^* + 1/2}{k^* - 1/2}\sqrt{\lambda} L_\gamma + \frac{k^*}{k^* - 1/2}\biggl(\frac{1}{2} L_\gamma^2\biggr).
$$
With this $k^*$ the deficit diverges,
$$
\lim_{|g| \to \infty} D(g; M, k^*) = -\infty.
$$
Since $D(g; M, k^*)$ is continuous on $\R$ and tends to $-\infty$ at both ends, it is non-positive outside some compact interval, and by the extreme value theorem it attains a finite global maximum $D_{max}(M, k^*) < \infty$.

Fix any $M > 0$. Since $r > 0$, choose $K$ with
$$
  K \ge \max\left(0, \frac{D_{max}(M, k^*)}{r}\right).
$$
Then $rK \ge D(g; M, k^*)$ for all $g$, which absorbs the deficit. Hence $\tilde{u}(g,t) := \alpha(t)\bigl(k^*\sqrt{g^2+M} + \frac{g}{2} + K\bigr)$ is a global supersolution in the $\mathcal{O}(|g|)$ uniqueness class.
\end{proof}
\begin{remark}
In the usual setting (e.g.\ BMS~\cite{BMS14}) supersolutions are built by gluing a quadratic function near the origin to a linear function at infinity. For state-dependent L\'evy jumps this is not convenient, since the non-local operator crosses the gluing boundary, and the infinite-activity singularity at $z=0$ requires global $C^2$ smoothness. A global quadratic barrier, on the other hand, would leave the $\mathcal{O}(|g|)$ uniqueness class and produce an $\mathcal{O}(|g|^4)$ jump variance. The function $f(g) = k\sqrt{g^2+M} + \frac{g}{2} + K$ avoids both problems: it is a single $C^2$ expression of order $\mathcal{O}(|g|)$, in the uniqueness class, whose second derivative decays as $\mathcal{O}(|g|^{-3})$ and thus controls the small-jump variance for large states.
\end{remark}

\subsection{Existence via Perron's Method}

\begin{theorem}\label{thm:exist-pide}
Under assumptions \eqref{eq:param} and condition \eqref{eq:growth_cond}, there exists a unique viscosity solution $u$ of \eqref{eq:bubble-pide} satisfying the barrier bounds $\underline{u} \leq u \leq \tilde u$.
\end{theorem}

\begin{proof}
We first check that the barriers are ordered, $\underline{u} \leq \tilde u$. Where $\psi(g,t) \leq 0$, $\underline{u}(g,t) = 0 \leq \tilde u(g,t)$. Where $\psi(g,t) > 0$, $\underline{u}(g,t) = \psi(g,t)$, and the obstacle condition for the supersolution gives $\tilde u(g,t) \geq \tilde u(-g,t) + \psi(g,t) \geq \psi(g,t) = \underline{u}(g,t)$. Hence the barriers are ordered.

Let $\mathcal{S}$ be the set of viscosity subsolutions between the barriers,
\[
  \mathcal{S} := \{w \mid w \text{ is a subsolution of \eqref{eq:bubble-pide}}, \text{ and } \underline{u} \leq w \leq \tilde u\}.
\]
This set is non-empty, since $\underline{u} \in \mathcal{S}$ by Lemma~\ref{lem:sub-pide}. Define the pointwise supremum
\[
  u(g,t) := \sup\{w(g,t) \mid w \in \mathcal{S}\}.
\]
By the stability of viscosity solutions under suprema (Proposition~\ref{prop:stab-pide}), the upper semicontinuous envelope $u^*$ is a subsolution. Since $\tilde u$ is continuous and bounds $\mathcal{S}$, we have $u^* \leq \tilde u$, so $u^* \in \mathcal{S}$ and $u^* = u$; thus $u$ is upper semicontinuous.

It remains to show that the lower semicontinuous envelope $u_*$ is a supersolution. Suppose, for contradiction, that $u_*$ fails the supersolution condition at an interior point $(g_0,t_0) \in \R \times [0,T)$.

In the non-local viscosity formulation of \cite[Definition $4$]{BI08}, the integro-differential operator is evaluated by splitting the measure at a radius $\delta > 0$. The interior integral ($|\gamma| < \delta$) is evaluated against a smooth test function, and the exterior integral ($|\gamma| \ge \delta$) against the envelope $u_*$.

Let $\varphi$ be a test function such that $u_* - \varphi$ has a local minimum of zero at $(g_0,t_0)$ in $B_\delta(g_0,t_0)$. To make the minimum strict, set $\hat{\varphi} := \varphi - \epsilon b$, where $b(g,t) \ge 0$ is a smooth quartic bump (for instance $|g-g_0|^4 + |t-t_0|^2$) that vanishes at the center and is positive elsewhere in $\bar{B}_\delta$; then $u_* - \hat{\varphi}$ has a strict local minimum at $(g_0,t_0)$. Since the first and second derivatives of $b$ vanish at $(g_0,t_0)$, the local operator $-\mathcal{P}\hat{\varphi}$ equals $-\mathcal{P}\varphi$. Subtracting $\epsilon b$ adds a small positive penalty to the negated interior integral, which is absorbed by taking $\epsilon > 0$ small. The failure of the supersolution property at $(g_0,t_0)$ means that the minimum is negative,
\[
  \min\bigl(-\cL^\nu_\delta(g_0,t_0; \hat{\varphi}, u_*), \;\; u_*(g_0,t_0) - u_*(-g_0,t_0) - \psi(g_0,t_0)\bigr) < 0.
\]
Hence at least one of the following holds for some $\eta > 0$,
\begin{align}
  -\cL^\nu_\delta(g_0,t_0; \hat{\varphi}, u_*) &\le -\eta < 0, \label{eq:fail1}\\
  u_*(g_0,t_0) - u_*(-g_0,t_0) - \psi(g_0,t_0) &\le -\eta < 0. \label{eq:fail2}
\end{align}

We first show that there is a gap $u_*(g_0,t_0) < \tilde{u}(g_0,t_0)$. If equality held, $\hat{\varphi}$ would touch the smooth upper barrier $\tilde{u}$ from below at $(g_0, t_0)$, so that $-\mathcal{P}\hat{\varphi} \ge -\mathcal{P}\tilde{u}$. Since $u_* \le \tilde{u}$, the exterior integral satisfies $\int_{|\gamma|\ge \delta} u_*\nu \le \int_{|\gamma|\ge \delta} \tilde{u}\nu$, hence $-\cL^\nu_\delta(g_0,t_0; \hat{\varphi}, u_*) \ge -\cL^\nu \tilde{u}(g_0,t_0)$. As $\tilde{u}$ is a classical supersolution, $-\cL^\nu \tilde{u} \ge 0$, contradicting \eqref{eq:fail1}; the obstacle case is similar. Thus there is a gap $m_1 := \tilde{u}(g_0,t_0) - \hat{\varphi}(g_0,t_0) > 0$.

Since the minimum of $u_* - \hat{\varphi}$ is strict in $\bar{B}_\delta$, the lower semicontinuous difference $u_* - \hat{\varphi}$ attains a strict minimum $m_2 > 0$ on the boundary $\partial B_\delta(g_0,t_0)$. Moreover, since $\nu$ is a Radon measure on $\R\setminus\{0\}$, we may choose the splitting radius $\delta > 0$ so that it excludes the origin and its boundary carries zero measure, $\nu(\{z \in \R\setminus\{0\} \mid |\gamma(g_0, z)| = \delta\}) = 0$; then the exterior measure is finite ($\nu_\delta < \infty$) and the integration domain depends continuously on the state.

We fix $0 < \mu < \min\bigl(m_1, \, m_2, \, \eta, \, \frac{\eta}{r + \nu_\delta}\bigr)$ and define the local modification
\[
  w(g,t) :=
  \begin{cases}
    \max\bigl(u(g,t),\,\hat{\varphi}(g,t)+\mu\bigr) & \text{inside } B_\delta(g_0,t_0),\\
    u(g,t) & \text{outside } B_\delta(g_0,t_0).
  \end{cases}
\]
By the gap $m_1$ and continuity, for $\delta$ small we have $\hat{\varphi}(g,t) + \mu \le \tilde{u}(g,t)$ in the neighborhood, so $w \le \tilde{u}$. The lower bound $w \ge \underline{u}$ holds since $w \ge u \ge \underline{u}$. Since $u_* - \hat{\varphi} \ge m_2 > \mu$ on $\partial B_\delta$, lower semicontinuity extends this inequality to a collar around $\partial B_\delta$, and since $u \ge u_*$, also $u > \hat{\varphi} + \mu$ there. Thus $w = u$ near the boundary, with no pasting singularity, and $w$ is upper semicontinuous.

We now check that the test branch $\tilde{\varphi} := \hat{\varphi} + \mu$ is a subsolution in $B_\delta$, in two cases.

\textbf{Case 1.} Suppose \eqref{eq:fail1} holds. Under the Barles--Imbert splitting~\cite[Definition $4$]{BI08}, jumps inside the neighborhood ($|\gamma| < \delta$) are evaluated against $\tilde{\varphi}$. The shift $\mu$ cancels in the compensated difference, $\tilde{\varphi}(g+\gamma) - \tilde{\varphi}(g) - \gamma\tilde{\varphi}'(g)\mathbf{1}_{\{|z|\le 1\}} = \hat{\varphi}(g+\gamma) - \hat{\varphi}(g) - \gamma\hat{\varphi}'(g)\mathbf{1}_{\{|z|\le 1\}}$, so the small-jump part is the same as for $\hat{\varphi}$.
Jumps outside the neighborhood ($|\gamma| \ge \delta$) are evaluated against $w$. Since $w \ge u_*$, evaluating the subtracted exterior integral on $u_*$ gives an upper bound. Subtracting the base value $\tilde{\varphi}(g_0) = \hat{\varphi}(g_0) + \mu$ leaves a shift penalty $+\mu \int_{|\gamma|\ge \delta} \nu = +\mu\nu_\delta$, and the discount term $-r\tilde{\varphi}$ adds $+r\mu$, so the total penalty is $\mu(r + \nu_\delta) < \eta$. At the center,
\[
  -\cL^\nu_\delta(g_0, t_0; \tilde{\varphi}, u_*) \le -\cL^\nu_\delta(g_0, t_0; \hat{\varphi}, u_*) + \mu(r+\nu_\delta) \le -\eta + \mu(r+\nu_\delta) < 0.
\]
Since $u_* \ge 0$ and the boundary of the integration domain has $\nu$-measure zero, Fatou's lemma applies to the exterior integral $\int_{|\gamma|\ge\delta} u_*(g+\gamma)\nu$, which is therefore lower semicontinuous in the moving domain. Hence $-\cL^\nu_\delta(\cdot; \tilde{\varphi}, u_*)$ is the sum of a continuous interior term and an upper semicontinuous exterior term, so it is upper semicontinuous; being negative at $(g_0,t_0)$, it is negative in a small ball $B_\delta$. Since $w \ge u_*$ gives $-\cL^\nu_\delta(\tilde{\varphi}, w) \le -\cL^\nu_\delta(\tilde{\varphi}, u_*)$, the operator on $w$ is negative there.

\textbf{Case 2.} Suppose \eqref{eq:fail2} holds. In $B_\delta$, whether $-g$ lies inside or outside the neighborhood, the construction gives $w(-g,t) \ge u_*(-g,t)$, so
\[
  \tilde{\varphi}(g,t) - w(-g,t) - \psi(g,t) \le \hat{\varphi}(g,t) + \mu - u_*(-g,t) - \psi(g,t).
\]
This is an upper semicontinuous expression, negative at $(g_0,t_0)$, hence negative throughout $B_\delta$.

In either case one branch of the minimum is negative, so $w$ is a subsolution, $w \in \mathcal{S}$.

On the other hand, since $u_*$ is the lower semicontinuous envelope of $u$, there is a sequence $(g_n, t_n) \to (g_0, t_0)$ with $u(g_n, t_n) \to u_*(g_0, t_0)$. As $w \in \mathcal{S}$ and $u = \sup \mathcal{S}$, we have $u(g_n, t_n) \ge w(g_n, t_n)$, and in $B_\delta$,
\[
  u(g_n, t_n) \ge \hat{\varphi}(g_n, t_n) + \mu.
\]
Letting $n \to \infty$, by continuity of $\hat{\varphi}$, $u_*(g_0, t_0) \ge u_*(g_0, t_0) + \mu$. Since $u_*$ is finite, this gives $0 \ge \mu$, contradicting $\mu > 0$. Hence $u_*$ is a supersolution on $\R \times [0,T)$.

Finally, we apply the comparison principle (Theorem~\ref{thm:comp-pide}). At $t=T$, $\alpha(T)=0$, so $\psi(g,T) = -c$, and both barriers vanish, $\underline{u}(g,T) = \max(0,-c) = 0$ and $\tilde{u}(g,T) = 0$. Since $\underline{u} \le u_* \le u^* \le \tilde{u}$, we get $u^*(g,T) = u_*(g,T) = 0$, the terminal condition. The comparison principle then gives $u^* \leq u_*$, and since $u_* \leq u^*$ by definition, $u^* = u_* = u$. Hence $u$ is continuous and is the unique viscosity solution.
\end{proof}
\begin{example}[Kou Jump-Diffusion Process]
As an example with state-dependent jump amplitude ($L_\gamma \neq 0$) satisfying the growth condition \eqref{eq:growth_cond}, we take a variant of the Kou double exponential jump-diffusion model \cite{Kou2002}.

Assume the disagreement process is subject to shocks arriving as a Poisson process with constant finite total intensity $\lambda > 0$, with jump size $z$ drawn from an asymmetric double exponential distribution. The L\'evy measure is
\[
  \nu(dz) = \lambda \left( p \eta_1 e^{-\eta_1 z} \mathbf{1}_{\{z>0\}} + q \eta_2 e^{\eta_2 z} \mathbf{1}_{\{z<0\}} \right) dz,
\]
where $p, q \ge 0$ with $p+q=1$ represent the probabilities of upward and downward jumps, and $\eta_1, \eta_2 > 0$ control the decay of the jump tails.

First, we determine the intensity of large jumps $\lambda_{Levy} := \nu(\{z \in \R\setminus\{0\} \mid |z| > 1\})$. Integrating the L\'evy measure over the domain $|z| > 1$ yields
\[
  \lambda_{Levy}  = \lambda \left( \int_1^\infty p \eta_1 e^{-\eta_1 z} dz + \int_{-\infty}^{-1} q \eta_2 e^{\eta_2 z} dz \right) = \lambda (p e^{-\eta_1} + q e^{-\eta_2}).
\]

To reflect that larger bubbles undergo larger absolute corrections, we take the jump amplitude proportional to the current state
\[
  \gamma(g, z) = c \cdot g \cdot z,
\]
for a scaling constant $c > 0$. Since the amplitude depends on $g$, the Lipschitz constant is non-zero ($L_\gamma \neq 0$).
We compute it from the integrated second moment of the state difference against the entire Kou measure. For any $x, y \in \R$,
\begin{multline*}
  \int_{\R\setminus\{0\}} |\gamma(x,z) - \gamma(y,z)|^2 \nu(dz) 
  = \int_{\R\setminus\{0\}} c^2 |x-y|^2 z^2 \nu(dz) =\\
  = c^2 |x-y|^2 \lambda \int_{\R\setminus\{0\}} z^2 \left( p \eta_1 e^{-\eta_1 z} \mathbf{1}_{\{z>0\}} + q \eta_2 e^{\eta_2 z} \mathbf{1}_{\{z<0\}} \right) dz.
\end{multline*}
By the standard moments of the exponential distribution, the second moment equals $\frac{2p}{\eta_1^2} + \frac{2q}{\eta_2^2}$, therefore
\[
  L_\gamma = c \sqrt{2\lambda \left( \frac{p}{\eta_1^2} + \frac{q}{\eta_2^2} \right)}.
\]
Substituting $\|\lambda\|_\infty$ and $L_\gamma$ into the growth condition \eqref{eq:growth_cond} gives
\[
  r + \rho > \sqrt{\lambda (p e^{-\eta_1} + q e^{-\eta_2})} \Biggl[ c \sqrt{2\lambda \left( \frac{p}{\eta_1^2} + \frac{q}{\eta_2^2} \right)} \Biggr] +  \frac{1}{2}\Biggl[ c^2 (2\lambda) \left( \frac{p}{\eta_1^2} + \frac{q}{\eta_2^2} \right) \Biggr],
\]
that is, after simplification we have
\[
  r + \rho > \sqrt{2} c \lambda \sqrt{(p e^{-\eta_1} + q e^{-\eta_2})\left( \frac{p}{\eta_1^2} + \frac{q}{\eta_2^2} \right)} +  c^2 \lambda \left( \frac{p}{\eta_1^2} + \frac{q}{\eta_2^2} \right).
\]
Since $r > 0$ and $\rho \ge 0$ are fixed, this inequality defines a non-empty parameter range; it holds, for instance, when the jump scale $c$ is small, or the tail parameters $\eta_1, \eta_2$ are large (so that both the jump variance and the large-jump intensity are small). In this range, the non-local obstacle PIDE \eqref{eq:bubble-pide} has a unique viscosity solution despite the state-dependent shocks.
\end{example}

\section{Convergence of the Numerical Scheme}\label{sec:convergence}

Having shown that the non-local obstacle PIDE \eqref{eq:bubble-pide} has a unique viscosity solution, we now study its numerical approximation. To prove that the discrete scheme converges to this solution, we use the framework of Barles and Souganidis \cite{BarlesSouganidis1991}. Since that framework is stated for bounded domains and local operators, we extend it to state-dependent L\'evy jumps by the Barles--Imbert splitting \cite{BI08}, and to the linear growth setting on an unbounded domain.

\subsection{Explicit Discretization Scheme}

Let the domain be truncated and discretized on a uniform spatio-temporal grid $\mathcal{G}_h = \{(g_i, t_n) \in \mathbb{R} \times [0, T] \mid g_i = i\Delta g, t_n = n\Delta t\}$, where $h = (\Delta g, \Delta t)$ denotes the discretization parameters and $\Delta t = T/N$. We denote the numerical approximation of the bubble value function at a node $(g_i, t_n)$ by $u_i^n := u^h(g_i, t_n)$.

We march backward in time from maturity $T$, with the discrete terminal condition $u_i^N = 0$. Since the non-local integral operator couples all grid nodes, a fully implicit discretization would require inverting a dense matrix at each step. To keep the scheme stable and efficient, we use the Implicit--Explicit scheme for jump-diffusions (see, e.g., Cont and Voltchkova \cite{ContVoltchkova2005}): the local differential terms are taken implicitly and the non-local integral explicitly.

So that the implicit matrix has the M-matrix property, which is needed for monotonicity and convergence \cite{dHalluin2005}, we split the drift by its direction (upwind), $g_i = g_i^+ - g_i^-$ with $g_i^+ := \max(g_i, 0)$ and $g_i^- := \max(-g_i, 0)$. The discrete operator $S_{\mathcal{L}}$ is
\begin{multline}\label{eq:discrete_operator}
S_{\mathcal{L}}(h, g_i, t_n, u_i^n, u) := \frac{u_i^n - u_i^{n+1}}{\Delta t}  + \rho g_i^+ \frac{u_i^n - u_{i-1}^n}{\Delta g} - \rho g_i^- \frac{u_{i+1}^n - u_i^n}{\Delta g} -\\ - \frac{\sigma^2}{2} \frac{u_{i+1}^n - 2u_i^n + u_{i-1}^n}{\Delta g^2} + r u_i^n - \mathcal{I}^h_i[u^{n+1}],
\end{multline}
where $\mathcal{I}^h_i$ is a consistent quadrature of the L\'evy measure with nodes $z_k$ and positive weights $w_k > 0$, over the truncated exterior domain,
\begin{multline}
\mathcal{I}^h_i[u^{n+1}] := \sum_k w_k \Bigl[ u(g_i + \gamma_k, t_{n+1}) - u_i^{n+1} - \Bigl( \gamma_k^+ \frac{u_i^{n+1} - u_{i-1}^{n+1}}{\Delta g} - \gamma_k^- \frac{u_{i+1}^{n+1} - u_i^{n+1}}{\Delta g} \Bigr) \mathbf{1}_{\{|z_k|\le 1\}} \Bigr].
\end{multline}

\begin{remark}
As noted by Jakobsen and Karlsen \cite{JakobsenKarlsen2005} for control problems with jumps, a central difference for the small-jump compensator can produce positive off-diagonal coefficients and break the elliptic structure of the scheme. To keep monotonicity, the compensator is discretized by the first-order upwind difference above, with $\gamma_k^+ := \max(\gamma_k, 0)$ and $\gamma_k^- := \max(-\gamma_k, 0)$, where $\gamma_k := \gamma(g_i, z_k)$.
\end{remark}

To account for the early exercise feature, we write the free-boundary problem as a linear complementarity problem, and the full scheme uses the minimum operator,
\begin{equation}
S\bigl(h, g_i, t_n, u_i^n, u\bigr) := \min\Bigl( S_{\mathcal{L}}\bigl(h, g_i, t_n, u_i^n, u\bigr),\; u_i^n - u_{-i}^n - \psi_i^n \Bigr) = 0,
\end{equation}
where $u_{-i}^n := u(-g_i, t_n)$ represents the reflected valuation of the opposing investor group, and $\psi_i^n := \psi(g_i, t_n)$ is the execution payoff.

\subsection{Properties of the Scheme}

For the numerical solutions to converge to the viscosity solution, the discrete operator $S$ must be stable, monotone, and consistent.

\begin{definition}[Stability]
The numerical scheme is \emph{stable} if, for any sufficiently small discretization parameter $h$, there exists a constant $C > 0$ independent of $h$ such that the numerical solution possesses at most linear growth globally
\[
|u^h(g, t)| \le C(1 + |g|), \quad \forall (g,t) \in \mathcal{G}_h.
\]
\end{definition}

\begin{definition}[Monotonicity]
The numerical scheme is \emph{monotone} if, for any evaluation point $(g, t)$, any scalar $r \in \mathbb{R}$, and any two grid functions $u, v$ satisfying $u \le v$ globally on $\mathcal{G}_h$, the operator satisfies
\[
S(h, g, t, r, u) \ge S(h, g, t, r, v).
\]
\end{definition}

\begin{definition}[Consistency]
The scheme is \emph{consistent} with the obstacle PIDE if, for any point $(g_0, t_0) \in \mathbb{R} \times [0, T)$ and any smooth test function $\varphi \in C^{2,1}(\mathbb{R} \times [0, T])$ possessing at most linear growth, the discrete operator limits satisfy
\begin{multline*}
\limsup_{\substack{h \to 0 \\ (g, t) \to (g_0, t_0) \\ \xi \to 0}} S\bigl(h, g, t, \varphi(g, t) + \xi, \varphi + \xi\bigr) \le \min\bigl(-\mathcal{L}^\nu \varphi(g_0, t_0),\; \varphi(g_0, t_0) - \varphi(-g_0, t_0) - \psi(g_0, t_0)\bigr),
\end{multline*}
and
\begin{multline*}
\liminf_{\substack{h \to 0 \\ (g, t) \to (g_0, t_0) \\ \xi \to 0}} S\bigl(h, g, t, \varphi(g, t) + \xi, \varphi + \xi\bigr) \ge \min\bigl(-\mathcal{L}^\nu \varphi(g_0, t_0),\; \varphi(g_0, t_0) - \varphi(-g_0, t_0) - \psi(g_0, t_0)\bigr).
\end{multline*}
\end{definition}
We now prove that our discretization has these properties.

\begin{lemma}[Monotonicity]\label{lem:monotonicity}
Assume the total discrete quadrature mass is finite, denoted by $\lambda_h := \sum_k w_k < \infty$. The  scheme $S$ is monotone provided the time step satisfies the Courant--Friedrichs--Lewy (CFL) condition
\begin{equation}\label{eq:cfl_condition}
\Delta t \le \min_{i} \Biggl( \sum_k w_k \biggl[ 1 + \frac{|\gamma(g_i, z_k)|}{\Delta g} \mathbf{1}_{\{|z_k|\le 1\}} \biggr] \Biggr)^{-1}.
\end{equation}
\end{lemma}

\begin{proof}
Following Cont and Voltchkova \cite{ContVoltchkova2005}, we write the operator $S_{\mathcal{L}}$ in matrix form. Let $U^m$ be the vector of grid values at time $t_m$. We split the operator into an implicit differential part $\mathcal{M}$ at the unknown time $t_n$ and an explicit integral part $\mathcal{N}$ at the known time $t_{n+1}$,
$$
S_{\mathcal{L}}\bigl(h, g_i, t_n, u_i^n, u\bigr) = [\mathcal{M} U^n]_i - [\mathcal{N} U^{n+1}]_i.
$$
The Barles--Souganidis monotonicity condition requires $S_{\mathcal{L}}$ to be non-increasing in all off-diagonal grid values. In matrix form this amounts to two properties: $\mathcal{M}$ is an M-matrix, with positive diagonal and non-positive off-diagonal entries, which gives unconditional stability of the local terms (see, e.g., \cite{BermanPlemmons1994}); and $\mathcal{N}$ is non-negative, $\mathcal{N}_{i,j} \ge 0$ for all $i,j$, so that the explicit step $-[\mathcal{N} U^{n+1}]_i$ enters with a non-positive sign.

Collecting the terms at $t_n$ in \eqref{eq:discrete_operator}, the $i$-th row of $\mathcal{M}$ is
\begin{multline*}
[\mathcal{M} U^n]_i = \left( \frac{1}{\Delta t} + \frac{\rho g_i^+}{\Delta g} + \frac{\rho g_i^-}{\Delta g} + \frac{\sigma^2}{\Delta g^2} + r \right) u_i^n - \left( \frac{\rho g_i^-}{\Delta g} + \frac{\sigma^2}{2\Delta g^2} \right) u_{i+1}^n - \left( \frac{\rho g_i^+}{\Delta g} + \frac{\sigma^2}{2\Delta g^2} \right) u_{i-1}^n.
\end{multline*}
Consider the off-diagonal entries of $\mathcal{M}$. Since $\sigma^2 > 0$ and $\Delta g > 0$, the central-difference diffusion weights are positive, and the upwind treatment of the drift gives $g_i^+ := \max(g_i,0) \ge 0$ and $g_i^- := \max(-g_i,0) \ge 0$.

Hence the off-diagonal entries $\mathcal{M}_{i, i+1} = -\bigl( \frac{\rho g_i^-}{\Delta g} + \frac{\sigma^2}{2\Delta g^2} \bigr)$ and $\mathcal{M}_{i, i-1} = -\bigl( \frac{\rho g_i^+}{\Delta g} + \frac{\sigma^2}{2\Delta g^2} \bigr)$ are negative. To see that $\mathcal{M}$ is a strictly diagonally dominant M-matrix, compare the sum of the off-diagonal absolute values with the diagonal,
$$\sum_{j \neq i} |\mathcal{M}_{i,j}| = |\mathcal{M}_{i,i+1}| + |\mathcal{M}_{i,i-1}| = \frac{\rho (g_i^+ + g_i^-)}{\Delta g} + \frac{\sigma^2}{\Delta g^2} = \frac{\rho |g_i|}{\Delta g} + \frac{\sigma^2}{\Delta g^2}.$$
Subtracting this from the diagonal $\mathcal{M}_{i,i}$ leaves the time-derivative and discount terms,
$$|\mathcal{M}_{i,i}| - \sum_{j \neq i} |\mathcal{M}_{i,j}| = \frac{1}{\Delta t} + r.$$
Since $\Delta t > 0$ and $r > 0$, this is positive ($\ge \frac{1}{\Delta t} > 0$). So $\mathcal{M}$ has positive diagonal, negative off-diagonal, and is strictly diagonally dominant. By standard matrix theory \cite{BermanPlemmons1994}, such a matrix is a non-singular M-matrix, so the implicit step is monotone.

Next we collect the explicit terms at $t_{n+1}$, namely the backward-Euler time derivative and the jump integral,
$$[\mathcal{N} U^{n+1}]_i = \frac{1}{\Delta t} u_i^{n+1} + \mathcal{I}_i^h[U^{n+1}].$$
Expanding the quadrature $\mathcal{I}_i^h$ and grouping by grid values, the entries of $\mathcal{N}$ are
\begin{multline*}
[\mathcal{N} U^{n+1}]_i = \left( \frac{1}{\Delta t} - \sum_k w_k - \sum_k w_k \frac{\gamma_k^+ + \gamma_k^-}{\Delta g} \mathbf{1}_{\{|z_k|\le 1\}} \right) u_i^{n+1} \\
+ \sum_k w_k u(g_i + \gamma_k, t_{n+1}) + \sum_k w_k \frac{\gamma_k^-}{\Delta g} \mathbf{1}_{\{|z_k|\le 1\}} u_{i+1}^{n+1} +\\ \sum_k w_k \frac{\gamma_k^+}{\Delta g} \mathbf{1}_{\{|z_k|\le 1\}} u_{i-1}^{n+1}.
\end{multline*}
The quadrature weights are positive ($w_k > 0$), and the upwind compensator gives $\gamma_k^+ \ge 0$, $\gamma_k^- \ge 0$, so the off-diagonal entries of $\mathcal{N}$ are non-negative.

For $\mathcal{N}$ to be non-negative, the diagonal entry must also be non-negative. Using $\gamma_k^+ + \gamma_k^- = |\gamma_k|$,
$$\mathcal{N}_{i,i} = \frac{1}{\Delta t} - \sum_k w_k \biggl[ 1 + \frac{|\gamma_k|}{\Delta g} \mathbf{1}_{\{|z_k|\le 1\}} \biggr] \ge 0.$$
Solving for $\Delta t$ and taking the minimum over the grid gives the CFL condition \eqref{eq:cfl_condition}. Under it, $\mathcal{N} \ge 0$ and the explicit step is monotone.

Finally, consider the obstacle term $S_{obs}(r, u) := r - u_{-i}^n - \psi_i^n$ with $r = u_i^n$. Monotonicity requires $S_{obs}(r, u) \ge S_{obs}(r, v)$ whenever $u \le v$. For $i \neq 0$, $S_{obs}(r, u) - S_{obs}(r, v) = v_{-i}^n - u_{-i}^n \ge 0$, as required. At the origin ($i = 0$) the reflection maps to the central node, giving $r - r - \psi(0,t) = c$; since $c > 0$ this branch is a positive constant, and as the scheme takes $\min(S_{\mathcal{L}}, S_{obs}) = 0$, it is inactive at the origin. Since both $S_{\mathcal{L}}$ and $S_{obs}$ are monotone, so is $S$.
\end{proof}

\begin{lemma}[Consistency]\label{lem:consistency}
The numerical operator $S$ is consistent with the non-local PIDE \eqref{eq:bubble-pide}.
\end{lemma}
\begin{proof}
Let $\xi \in \mathbb{R}$ be a constant shift. We evaluate the scheme at the test function $\varphi + \xi$. The shift $\xi$ cancels in all spatial and temporal finite differences.

Taylor expansion at $(g_i, t_n)$ in \eqref{eq:discrete_operator} gives the local truncation errors: the backward time derivative, upwind drift, and central diffusion give errors of order $\mathcal{O}(\Delta t)$, $\mathcal{O}(\Delta g)$, and $\mathcal{O}(\Delta g^2)$.
For the jump integral, quadrature of $\nu(dz)$ against a smooth function of linear growth gives a truncation error $\mathcal{O}(\Delta g^\kappa)$, with $\kappa$ depending on the quadrature rule and the L\'evy tails,
\[
-\mathcal{I}^h_i[\varphi^{n+1}] = -\int_{\mathbb{R}_0} \Bigl[ \varphi(g_i+\gamma, t_n) - \varphi - \gamma\varphi_g \mathbf{1}_{\{|z|\le 1\}} \Bigr]\nu(dz) + \mathcal{O}(\Delta t) + \mathcal{O}(\Delta g^\kappa).
\]
The only place where the shift $\xi$ survives is the discount term, $r(\varphi_i^n + \xi) = r\varphi_i^n + r\xi$. Summing the terms,
\[
S_{\mathcal{L}}(h, g_i, t_n, \varphi_i^n+\xi, \varphi+\xi) = -\mathcal{L}^\nu \varphi(g_i, t_n) + r\xi + \mathcal{E}(h),
\]
with $\lim_{h \to 0} \mathcal{E}(h) = 0$. Letting $h \to 0$ and $\xi \to 0$ gives $-\mathcal{L}^\nu \varphi$.

For the obstacle branch, $\varphi + \xi$ gives $\varphi_i^n - \varphi_{-i}^n - \psi_i^n$, in which the shift cancels. Since $\min(\cdot, \cdot)$ is Lipschitz, the limit of the full operator is the continuous expression in the consistency definition.
\end{proof}

\begin{lemma}[Stability]\label{lem:stability}
Under the strict CFL condition established in Lemma \ref{lem:monotonicity}, there exists a constant $C > 0$ independent of $h$ such that the discrete solution satisfies $0 \le u_i^n \le C(1 + |g_i|)$ for all $(g_i, t_n) \in \mathcal{G}_h$.
\end{lemma}

\begin{proof}
We use the global discrete comparison principle via barrier functions. Let $U^n = \{u_i^n\}_{i=-J}^J$ denote the numerical solution vector at time $t_n$ on a finite, symmetric computational grid $\mathcal{G}_h = -\mathcal{G}_h$. Let $\tilde{U}$ be the continuous global supersolution defined by
$$
\tilde{u}(g, t) := \alpha(t)(k\sqrt{g^2+M} + g/2 + K).
$$
Due to $\tilde{u} \in C^{2,1}$ grows at most linearly, its spatial gradient $\tilde{u}_g$ is globally bounded, and its higher-order spatial derivatives ($\tilde{u}_{gg}, \tilde{u}_{ggg}$) decay asymptotically. Therefore, the Taylor truncation error associated with consistency (Lemma \ref{lem:consistency}) is uniformly bounded. We therefore have
\[
S(h, g_i, t_n, \tilde{U}^n, \tilde{U}^{n+1}) = \min\bigl(-\mathcal{L}^\nu \tilde{u}(g_i, t_n), \tilde{u}_i^n - \tilde{u}_{-i}^n - \psi_i^n\bigr) + \mathcal{O}(h).
\]
By choosing the constant $K$ strictly larger than the minimal requirement of the continuous deficit function, the continuous operator $-\mathcal{L}^\nu \tilde{u}$ is uniformly positive. Furthermore, the test function yields exactly $\tilde{u}(g,t) - \tilde{u}(-g,t) = \alpha(t)g$. Because the grid is symmetric, the index $-i$ is exactly on the grid, ensuring the discrete obstacle gap evaluates algebraically to precisely $\alpha(t)g_i - (\alpha(t)g_i - c) = c > 0$ globally. Thus, there exists a sufficiently small grid spacing $h_0$ such that for all $h < h_0$, $S(h, g_i, t_n, \tilde{U}^n, \tilde{U}^{n+1}) \ge 0$ at all interior nodes. 

We argue by backward induction in time. At the terminal step $t_N = T$, $U^N = 0 \le \tilde{U}^N$. Assume the bound holds at the forward step,i.e. $U^{n+1} \le \tilde{U}^{n+1}$. The numerical scheme solves $S(h, g_i, t_n, U^n, U^{n+1}) = 0$. The strict CFL condition established in Lemma \ref{lem:monotonicity} ensures that all explicit forward-time weights in the scheme are non-negative. Thus, the operator $S$ is non-increasing with respect to the forward vector $U^{n+1}$. Replacing $U^{n+1}$ with the larger supersolution $\tilde{U}^{n+1}$  decreases the operator value
\begin{equation}\label{eq:induct_step}
S(h, g_i, t_n, U^n, \tilde{U}^{n+1}) \le S(h, g_i, t_n, U^n, U^{n+1}) = 0.
\end{equation}
To deduce $U^n \le \tilde{U}^n$ globally, assume for contradiction that the maximum positive difference $\delta := \max_j (u_j^n - \tilde{u}_j^n)$ is strictly positive ($\delta > 0$). Assuming compatible numerical bounds at the truncated edges $j = \pm J$, this strictly positive maximum must be attained at some interior node $i \in (-J, J)$. Thus, $u_i^n = \tilde{u}_i^n + \delta$, and for all other nodes, $u_j^n \le \tilde{u}_j^n + \delta$.

We evaluate the discrete operator components at this interior maximum node $i$. By the M-matrix properties of the scheme, $S_{\mathcal{L}}$ is non-increasing with respect to the off-diagonal terms $u_j^n$ ($j \neq i$). Replacing the off-diagonal terms with their upper bounds $\tilde{u}_j^n + \delta$ strictly maintains or decreases the operator value
\[
S_{\mathcal{L}}(U^n, \tilde{U}^{n+1})_i \ge S_{\mathcal{L}}(\tilde{U}^n + \delta \mathbf{1}, \tilde{U}^{n+1})_i.
\]
Because the linear operator $S_{\mathcal{L}}$ discretizes spatial derivatives and jump integrals as conservative linear combinations of differences, shifting the current-time vector by a global constant $\delta$ mathematically annihilates the discrete spatial and integral operators perfectly. The only surviving terms are the zero-order coefficients: the backward time-difference $\frac{1}{\Delta t}u_i^n$ and the strictly positive discount rate $r u_i^n$. This exact algebraic separation yields
\[
S_{\mathcal{L}}(\tilde{U}^n + \delta \mathbf{1}, \tilde{U}^{n+1})_i = S_{\mathcal{L}}(\tilde{U}^n, \tilde{U}^{n+1})_i + \left(\frac{1}{\Delta t} + r\right)\delta > S_{\mathcal{L}}(\tilde{U}^n, \tilde{U}^{n+1})_i.
\]
Since $\tilde{U}^n$ is an interior supersolution, then $S_{\mathcal{L}}(\tilde{U}^n, \tilde{U}^{n+1})_i \ge 0$ implies $S_{\mathcal{L}}(U^n, \tilde{U}^{n+1})_i > 0$.

Next, we evaluate the obstacle constraint at the interior maximum node $i$. The obstacle operator is strictly non-increasing with respect to the off-diagonal term $u_{-i}^n$ due to its negative coefficient. Substituting the exact value $u_i^n = \tilde{u}_i^n + \delta$ and the upper bound $u_{-i}^n \le \tilde{u}_{-i}^n + \delta$, the uniform shift $\delta$ cancels perfectly 
\[
S_{obs}(U^n)_i \ge (\tilde{u}_i^n + \delta) - (\tilde{u}_{-i}^n + \delta) - \psi_i^n = \tilde{u}_i^n - \tilde{u}_{-i}^n - \psi_i^n = c > 0.
\]
Since both components are strictly positive, their minimum is strictly positive, meaning  $$S(h, g_i, t_n, U^n, \tilde{U}^{n+1}) > 0.$$
However, this  contradicts the induction bound \eqref{eq:induct_step}, which requires $S \le 0$ at all interior nodes. Therefore, $\delta > 0$ must be false, proving $U^n \le \tilde{U}^n$ globally.

For the lower bound, we utilize the trivial smooth barrier $\underline{U} \equiv 0$. The continuous linear operator evaluates to exactly zero: $S_{\mathcal{L}}(0,0) = 0$. The minimum of zero and any obstacle value is unconditionally non-positive, ensuring $S(h, g_i, t_n, 0, 0) \le 0$ uniformly. By identical backward induction logic (driving a contradiction via the negative shift $\left(\frac{1}{\Delta t} + r\right)\delta < 0$), we  provide $U^n \ge 0$ everywhere.

 By definition, the numerical scheme  enforces $\min\bigl(S_{\mathcal{L}}(U^n), \, u_i^n - u_{-i}^n - \psi_i^n\bigr) = 0$. This universally implies $u_i^n - u_{-i}^n - \psi_i^n \ge 0$. Because we established $u_{-i}^n \ge 0$, it strictly follows that $u_i^n \ge \psi_i^n$. 
Since  $\alpha(t)$ is non-increasing, then the test function is strictly bounded by 
$$
\tilde{u}(g_i, t_n) \le \alpha(0)(k(|g_i|+\sqrt{M}) + |g_i|/2 + K) \le C(1+|g_i|),
$$
for a sufficiently large finite constant $C$. Therefore, the numerical solution enjoys the exact bounds $\max(0, \psi_i^n) \le u_i^n \le C(1+|g_i|)$. This confirms the unconditional $L^\infty$-weighted stability of the scheme.
\end{proof}

\subsection{Convergence Result}

Having shown that the above scheme is stable, monotone, and consistent, we now prove the convergence result, which links the discrete approximations to the viscosity framework of Sections \ref{sec:viscosity} and \ref{sec:comparison}.

\begin{theorem}[Convergence]\label{thm:convergence_final}
Assume the parameters of the L\'evy jump-diffusion satisfy \eqref{eq:param} and the growth condition \eqref{eq:growth_cond}. As $h \to 0$, the discrete approximation $u^h$ converges locally uniformly on $\mathbb{R} \times [0, T]$ to the unique viscosity solution $u$ of the non-local PIDE \eqref{eq:bubble-pide}.
\end{theorem}
\begin{proof}
Define the upper and lower relaxed semicontinuous envelopes of the numerical solutions,
\[
\bar{u}(g, t) := \limsup_{\substack{h \to 0 \\ (g', t') \to (g, t)}} u^h(g', t'), \quad \underline{u}(g, t) := \liminf_{\substack{h \to 0 \\ (g', t') \to (g, t)}} u^h(g', t').
\]
By the stability of Lemma \ref{lem:stability}, $\bar{u}$ and $\underline{u}$ are finite and lie in the linear growth class $\mathcal{O}(|g|)$, and $\underline{u}(g, t) \le \bar{u}(g, t)$ by definition.

Let $\varphi \in C^{2,1}$, and let $(g_0, t_0)$ with $t_0 < T$ be a strict global maximum of $\bar{u} - \varphi$. Without loss of generality, $\bar{u}(g_0, t_0) = \varphi(g_0, t_0)$ and $\bar{u} \le \varphi$.

By the standard properties of relaxed limits, there are $h_n \to 0$ and points $(g_n, t_n) \to (g_0, t_0)$ such that $(g_n, t_n)$ is a global maximum of $u^{h_n} - \varphi$ and $u^{h_n}(g_n, t_n) \to \bar{u}(g_0, t_0)$. Set $\xi_n := u^{h_n}(g_n, t_n) - \varphi(g_n, t_n) \to 0$. The maximum property gives $u^{h_n} \le \varphi + \xi_n$ on the grid $\mathcal{G}_{h_n}$.
Since $u^{h_n}$ solves the scheme
\[
S\bigl(h_n, g_n, t_n, u^{h_n}(g_n, t_n), u^{h_n}\bigr) = 0.
\]
To handle the non-local operator we use the Barles--Imbert splitting \cite{BI08}. Since the integral evaluates the function over the whole domain, the local bound by $\varphi$ is not enough. Fix a splitting radius $\delta > 0$ such that the boundary of $B_\delta(g_0, t_0)$ carries zero L\'evy measure. Define $w^{h_n}$ to equal $\varphi + \xi_n$ inside $B_\delta(g_n, t_n)$ and $u^{h_n}$ outside. Since $u^{h_n} \le \varphi + \xi_n$, we have $u^{h_n} \le w^{h_n}$.

By monotonicity (Lemma \ref{lem:monotonicity}) and $u^{h_n} \le w^{h_n}$,
\[
S\bigl(h_n, g_n, t_n, \varphi(g_n, t_n) + \xi_n, w^{h_n}\bigr) \le 0.
\]

We take the $\liminf$ as $n \to \infty$. Inside $B_\delta$ the operator evaluates $\varphi$, and by consistency (Lemma \ref{lem:consistency}) the local terms converge.

Outside $B_\delta$ the quadrature evaluates $u^{h_n}$. Since the exterior L\'evy measure is finite and bounded away from the origin, the discrete sum converges to the continuous integral. By the reverse Fatou lemma, with the linear growth bound of Lemma \ref{lem:stability} as dominating function,
\[
\limsup_{n \to \infty} \sum_{|\gamma(g_n, z_k)| \ge \delta} w_k \, u^{h_n}\bigl(g_n + \gamma(g_n, z_k), t_{n+1}\bigr) \le \int_{|\gamma| \ge \delta} \bar{u}(g_0 + \gamma, t_0) \nu(dz).
\]
Since the jump integral is subtracted in the negated operator $-\mathcal{L}^\nu$, this $\limsup$ becomes a $\liminf$ lower bound for the full operator, and the limit gives the Barles--Imbert test formulation \eqref{eq:testop},
\[
\min\bigl(\mathcal{L}^\nu_{\varphi, \delta}(g_0, t_0; \bar{u}),\; \bar{u}(g_0, t_0) - \bar{u}(-g_0, t_0) - \psi(g_0, t_0)\bigr) \le 0.
\]
Thus $\bar{u}$ is a subsolution of the PIDE.

By a symmetric argument with a strict local minimum and the standard Fatou lemma for the lower envelope, $\underline{u}$ is a supersolution.

Since both functions lie in the linear growth class, the comparison principle (Theorem \ref{thm:comp-pide}) gives $\bar{u} \le \underline{u}$ on $\mathbb{R} \times [0, T)$. As $\underline{u} \le \bar{u}$ by definition, $\bar{u}(g, t) = \underline{u}(g, t)$ everywhere.
Hence the discrete solutions $u^h$ converge locally uniformly to the continuous function $u := \bar{u} = \underline{u}$, which is the unique viscosity solution, i.e.\ the bubble value function.
\end{proof}

\section{Simulations}\label{sec:simulations}
    
    We illustrate the bubble value function $u(g,t)$ obtained by solving the
    non-local obstacle PIDE \eqref{eq:bubble-pide} numerically, and we test the
    sensitivity of the speculative premium to the choice of driving L\'evy process.
    We consider four specifications spanning finite and infinite jump activity: the
    Merton jump--diffusion, the Variance Gamma (VG) and Normal Inverse Gaussian
    (NIG) infinite-activity models, and the Kou double-exponential model.
    
    \subsection{Discretization and Solver}\label{sec:disc}
    
    The state variable is discretized on a uniform grid
    $g\in[g_{\min},g_{\max}]=[-2,2]$ with $M+1$ nodes, and calendar time on
    $[0,T]$ with $N+1$ nodes; we integrate \emph{backward} from the terminal
    condition $u(g,T)=0$, writing $\tau:=T-t$ for the time to maturity and
    $U^m$ for the vector of interior grid values at $\tau_m := m\,\Delta t$.
    The local spatial part of $-\mathcal{L}^\nu$, namely
    $\rho g\,u_g-\tfrac{\sigma^2}{2}u_{gg}+r\,u$, is discretized implicitly:
    the mean-reversion drift by an upwind difference (which preserves the
    M-matrix structure of the local operator), the diffusion by central
    differences, and the discount on the diagonal, giving a sparse local
    matrix $D$; the Dirichlet data at the edges of the window enter the
    diffusion stencil of the adjacent interior nodes. The non-local term is
    approximated by quadrature of the L\'evy
    density on a truncated symmetric jump grid
    $z\in[-z_{\max},z_{\max}]\setminus(-\epsilon,\epsilon)$ with $z_{\max}=3$ and
    $\epsilon=0.05$; the post-jump state $g+\gamma(g,z)$ is mapped back to the grid
    by linear interpolation, producing a jump matrix $J$ with non-negative
    off-diagonal entries. The jumps are
    \emph{state-dependent} and proportional to the disagreement level,
    $\gamma(g,z)=c_J\,g\,z$, exactly as in the Kou example of
    Section~\ref{sec:existence}. One step of the march solves the linear
    complementarity problem
    \[
      \min\bigl(A\,U^{m+1}-U^{m},\;U^{m+1}-\Phi^{m+1}\bigr)=0,\qquad
      A:=I+\Delta t\,(D-J),
    \]
    with the \emph{coupled} obstacle
    $\Phi^{m+1}(g_i)=u(-g_i,\tau_{m+1})+\psi(g_i,\tau_{m+1})$ evaluated at the
    reflected node and at the \emph{same} time level. Each linear
    complementarity problem is solved by Projected Successive Over-Relaxation
    (PSOR), and the time-march is wrapped in an
    outer Picard iteration that updates the reflected value $u(-g,\cdot)$, repeated
    until $\lVert U^{(k+1)}-U^{(k)}\rVert_\infty<10^{-8}$.

    \begin{remark}[Fully implicit jump treatment]
    In the implementation both the local matrix $D$ and the dense jump
    matrix $J$ are treated implicitly. The matrix $A$ has non-positive
    off-diagonal entries and its diagonal dominance excess is at least
    $1+\Delta t\,r$, so it is a non-singular M-matrix for \emph{every}
    $\Delta t$: the scheme is unconditionally monotone, and the convergence
    framework of Section~\ref{sec:convergence} applies with the explicit part
    of the splitting empty, i.e.\ without the CFL restriction
    \eqref{eq:cfl_condition}.
    \end{remark}
    
    \begin{remark}[Convergence of the coupled iteration]
    The reflected coupling makes the obstacle problem genuinely nonlinear, and
    two implementation details are essential. First, each backward sweep must
    use the \emph{current} sweep's adjacent time level when forming the
    right-hand side, and the outer loop must be iterated to tolerance rather
    than truncated at a fixed small count: an insufficiently converged
    iteration underestimates $u(0,0)$ by more than an order of magnitude while
    leaving the obstacle-dominated peak near $g=g_{\max}$ almost unchanged.
    Second, the reflected value in the obstacle must be taken at the
    \emph{same} time level $\tau_{m+1}$ (from the previous outer iterate).
    Lagging it by one time level still converges, but the limit then violates
    the coupled constraint $u\ge u(-g,\cdot)+\psi$ by an
    $\mathcal{O}(\Delta t)$ margin ($\approx 3\times10^{-4}$ at this
    calibration), which distorts any tolerance-based reading of the contact
    set. With both in place the scheme converges in six outer iterations to
    an update below $10^{-9}$ for every model (Figure~\ref{fig:convergence}).
    \end{remark}
    
    The obstacle uses the BMS time profile evaluated at the time to maturity
    \[
      \alpha(t)=\bar\alpha(T-t),\qquad
      \bar\alpha(\tau):=\frac{1-e^{-(r+\lambda)\tau}}{r+\lambda},
    \]
    so that $\alpha(T)=0$ and $\alpha'(t)\le 0$, in accordance with the
    standing assumptions of Section~\ref{sec:bubble-levy}, and
    $\psi(g,t)=g\,\alpha(t)-c$. The parameters are listed in
    Table~\ref{tab:parameters}; all jump measures are symmetric, so the omitted
    $|z|\le1$ compensator integral $\int\gamma\,\mathbf 1_{|z|\le1}\nu(dz)$ vanishes
    identically and introduces no drift bias.
    
    \begin{table}[htbp]
    \centering
    \begin{tabular}{llc}
    \toprule
    \textbf{Parameter} & \textbf{Symbol} & \textbf{Value} \\ \midrule
    Risk-free / discount rate    & $r$                    & $0.05$ \\
    Mean-reversion rate          & $\rho$                 & $1.0$ \\
    Diffusive volatility         & $\sigma$               & $0.25$ \\
    Resale intensity (in $\bar\alpha$) & $\lambda$        & $1.0$ \\
    Transaction cost             & $c$                    & $0.1$ \\
    Jump-amplitude scale ($\gamma=c_J\,g\,z$) & $c_J$     & $0.1$ \\
    Terminal time                & $T$                    & $3.0$ \\
    Merton jump mean \& vol      & $\mu_J,\ \sigma_J$     & $0.0,\ 0.1$ \\
    VG parameters                & $C,\ G,\ M_{\!VG}$     & $0.2,\ 5.0,\ 5.0$ \\
    NIG parameters               & $\alpha,\ \beta,\ \delta$ & $15.0,\ 0.0,\ 0.5$ \\
    Kou parameters               & $\lambda_J,\ p,\ \eta_1,\ \eta_2$ & $1.0,\ 0.5,\ 10,\ 10$ \\
    Time / space grid            & $N,\ M$                & $200,\ 300$ \\
    Computational domain         & $[g_{\min},g_{\max}]$  & $[-2.0,\,2.0]$ \\ \bottomrule
    \end{tabular}
    \caption{Model and grid parameters used in the numerical experiments.}
    \label{tab:parameters}
    \end{table}
    
    \subsection{The Four L\'evy Specifications}
    
    To test the scheme across different jump structures, we evaluate the non-local operator for four L\'evy specifications, of both finite and infinite activity.

    The finite-activity \textbf{Merton} model \cite{Merton1976} models rare macroeconomic shocks by normally distributed jumps. Its L\'evy density is bounded at the origin,
    \begin{equation*}
        \nu(dz)=\frac{\lambda}{\sqrt{2\pi\sigma_J^2}} \exp\!\left(-\frac{(z-\mu_J)^2}{2\sigma_J^2}\right)\,dz.
    \end{equation*}
    
    The finite-activity \textbf{Kou} model \cite{Kou2002} replaces the Gaussian jumps by an asymmetric double-exponential law, with heavier tails. Its density is
    \begin{equation*}
        \nu(dz)=\lambda\bigl(p\,\eta_1 e^{-\eta_1 z}\mathbf 1_{\{z>0\}} +q\,\eta_2 e^{\eta_2 z}\mathbf 1_{\{z<0\}}\bigr)\,dz, \quad \text{with } p+q=1.
    \end{equation*}
    
    Among infinite-activity models, the \textbf{Variance Gamma (VG)} model \cite{Madan1998} produces infinitely many small jumps, and its density has a non-integrable singularity at the origin,
    \begin{equation*}
        \nu(dz)=C\bigl(e^{G z}|z|^{-1}\mathbf 1_{\{z<0\}} +e^{-M_{\!VG} z}z^{-1}\mathbf 1_{\{z>0\}}\bigr)\,dz.
    \end{equation*}
    
    The \textbf{Normal Inverse Gaussian (NIG)} model \cite{BarndorffNielsen1997} also has infinite activity, with semi-heavy tails, and gives a further test of the scheme for a singular kernel,
    \begin{equation*}
        \nu(dz)=\frac{\delta\alpha}{\pi}\,e^{\beta z}\,|z|^{-1}K_1(\alpha|z|)\,dz,
    \end{equation*}
    where $K_1$ is the modified Bessel function of the second kind.
    
    As Figure~\ref{fig:densities} shows, these four measures have quite different structure. The finite-activity Merton and Kou densities are bounded, whereas the infinite-activity VG and NIG densities diverge as $z \to 0$. Since the singular densities are not integrable at the origin, the quadrature excludes the shaded region ($|z|\le\epsilon$); the discarded small jumps are not otherwise compensated, and the effect of this truncation is quantified in the discussion at the end of the section.
    
    \begin{figure}[htbp]
    \centering
    \includegraphics[width=0.8\textwidth]{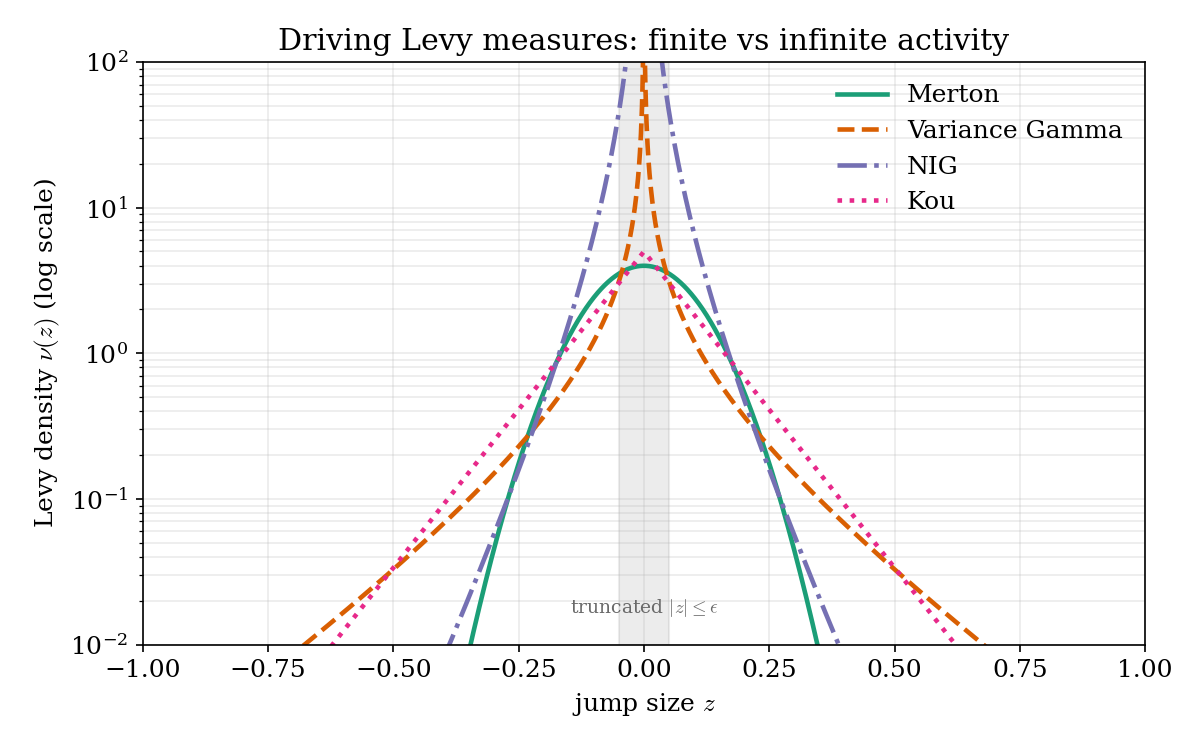}
    \caption{The four driving L\'evy densities $\nu(z)$ on a logarithmic scale. The Merton and Kou densities are bounded at the origin (finite activity); the VG and NIG densities diverge at $z=0$ (infinite activity). The shaded region indicates the domain $|z|\le\epsilon$ excluded from the quadrature.}
    \label{fig:densities}
    \end{figure}
    \subsection{The Bubble Surface and the Free Boundary}
    
    Figure~\ref{fig:surface} shows the computed surface $u_h(g,t)$ (left) and its
    contour map (right) for the representative Merton model. The solution is
    essentially flat and close to zero for $g<0$ and rises monotonically for $g>0$,
    peaking at the upper boundary with $u_h(g_{\max},0)\approx 1.7231$. Only the
    optimistic side ($g>0$) carries a positive resale premium. The exercise
    region is the contact set $\{u = u(-g,\cdot)+\psi\}$; at this calibration
    it is the interval $[g^*(\tau),\,g_{\max}]$, and the optimal exercise
    threshold $g^*(\tau)$ --- the lower edge of the contact set --- is plotted
    in Figure~\ref{fig:free_boundary}. The threshold \emph{decreases} as the
    horizon lengthens, from $g^*\approx 0.75$ at $\tau=0.15$ to a saturation
    value $g^*\approx 0.24$ for $\tau\gtrsim 1.5$: a longer resale horizon
    raises the time profile $\bar\alpha(\tau)$, so reselling becomes optimal
    at smaller levels of disagreement. This mirrors the behaviour of the
    uncoupled threshold $c/\bar\alpha(\tau)$ at which the payoff activates
    (it falls from $0.68$ to $0.11$ over the same range), shifted upward by
    the resale coupling $u(-g,\cdot)$.
    
    \begin{figure}[!htbp]
    \centering
    \includegraphics[width=\textwidth]{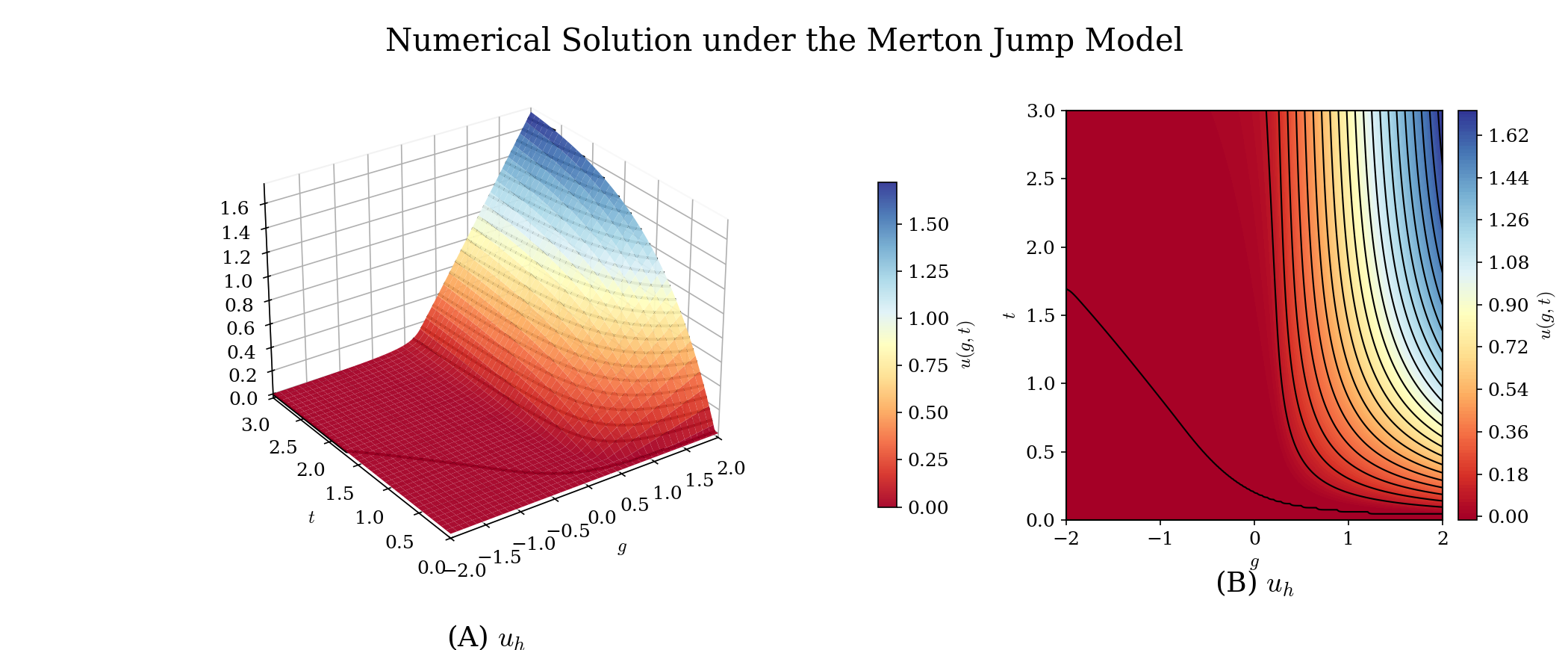}
    \caption{Representative bubble surface (Merton model). Left: $u_h(g,t)$ with the
    vertical axis denoting $\tau=T-t$. Right: contour map. The surfaces for the VG,
    NIG and Kou models are indistinguishable from this one (see
    Section~\ref{sec:robust}).}
    \label{fig:surface}
    \end{figure}
    
    \begin{figure}[!htbp]
    \centering
    \includegraphics[width=\textwidth]{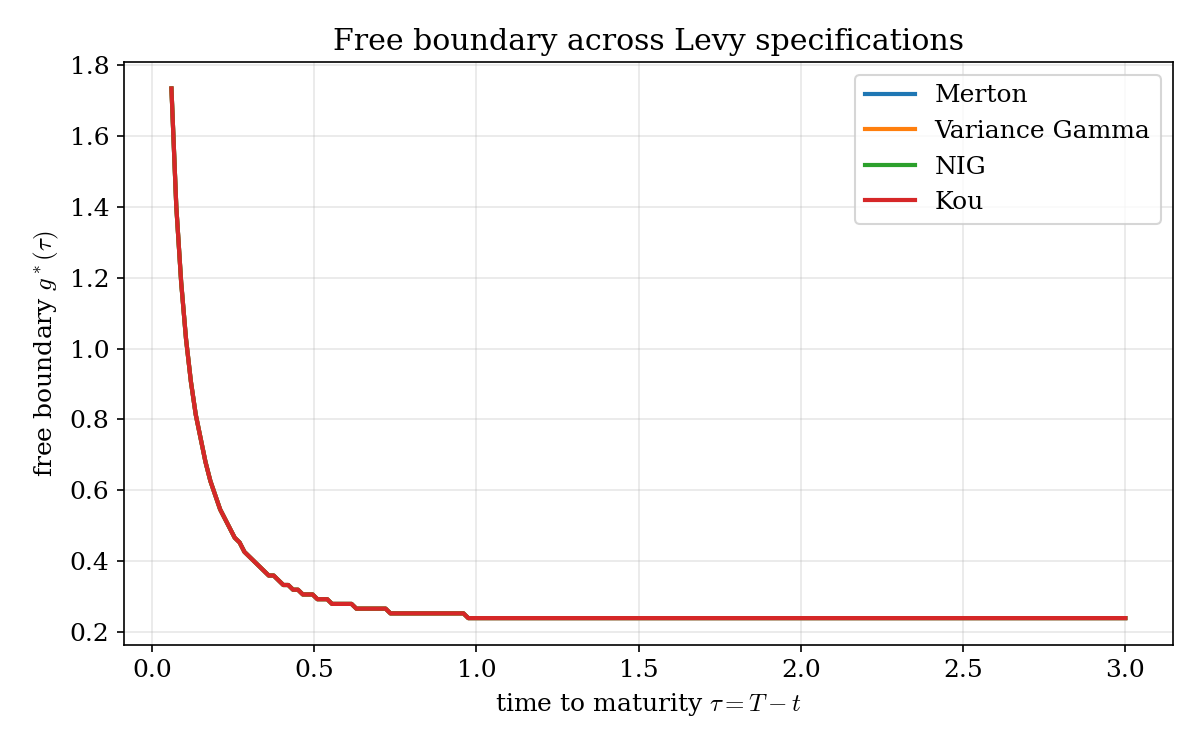}
    \caption{Optimal exercise threshold $g^*(\tau)$, computed as the lower edge
    of the discrete contact set $\{u=u(-g,\cdot)+\psi\}$. The curves for all four L\'evy
    specifications coincide to plotting accuracy.}
    \label{fig:free_boundary}
    \end{figure}
    
    \subsection{Robustness Across the L\'evy Specification}\label{sec:robust}
    
    Despite the pronounced differences between the driving measures
    (Figure~\ref{fig:densities}), the resulting bubble values are almost identical.
    Figure~\ref{fig:slice}(top) overlays the four value functions at $t=0$; the
    curves lie on top of one another. The bottom panel plots the maximum model
    spread $\max_i|u_i(g,0)-u_{\mathrm{Merton}}(g,0)|$ on a logarithmic scale:
    it peaks at only $\sim1.9\times10^{-5}$, at $g\approx0.16$ --- between the
    activation point of the payoff, $c/\bar\alpha(T)\approx 0.11$, and the
    saturated free boundary $g^*\approx 0.24$, where the solution has the
    least regularity --- and is orders of magnitude smaller elsewhere. The full
    space--time structure of the discrepancy is shown in
    Figure~\ref{fig:heatmap}: the disagreement between any model and Merton is
    concentrated in a band around the free boundary and grows
    mildly with the horizon $\tau$, but never exceeds the magnitudes in
    Table~\ref{tab:diffs}.
    
    \begin{figure}[htbp]
    \centering
    \includegraphics[width=\textwidth]{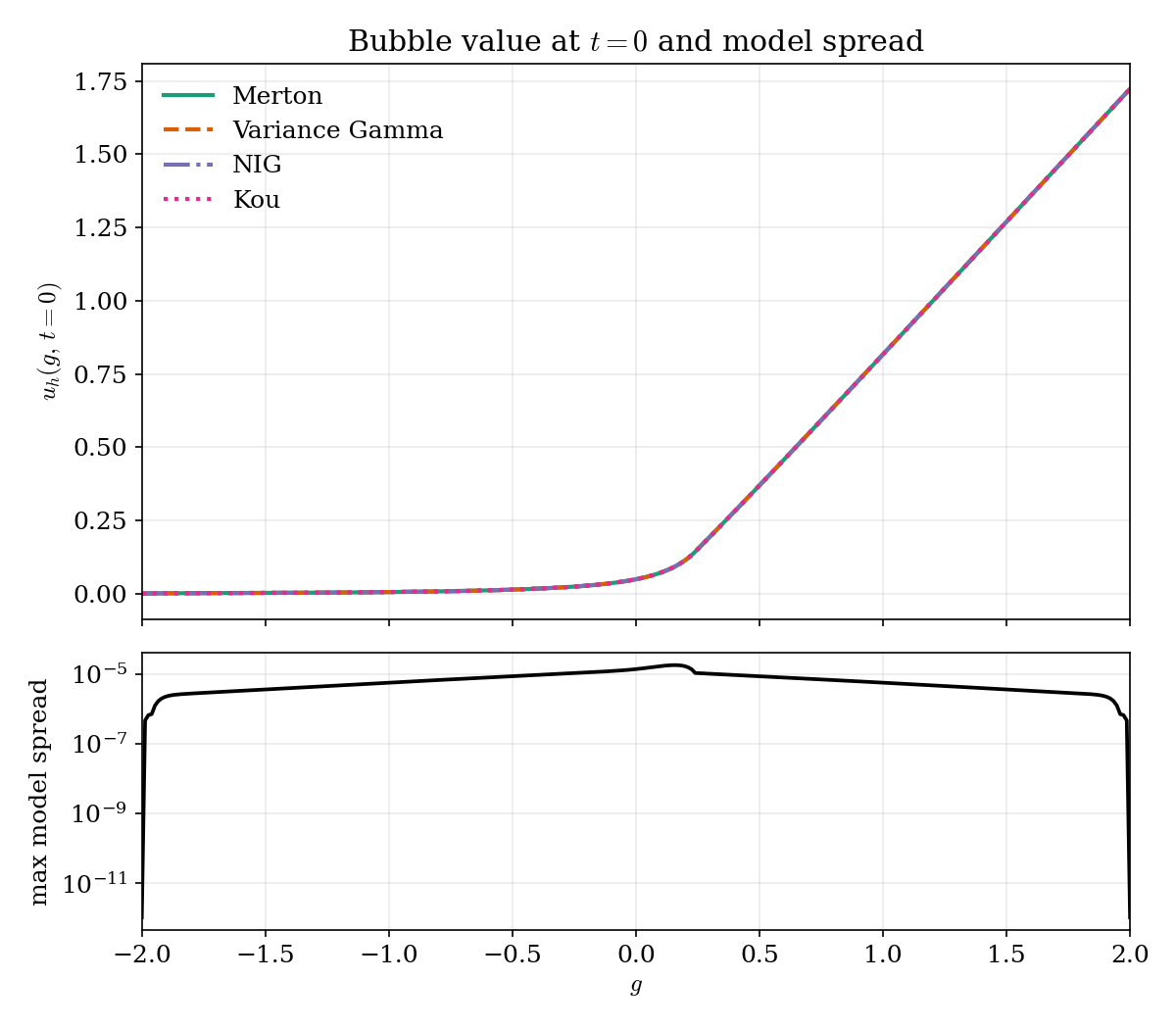}
    \caption{Top: bubble value $u_h(g,0)$ for the four models (curves coincide).
    Bottom: the maximum model spread (log scale), peaking at $\sim1.9\times10^{-5}$
    near the free boundary.}
    \label{fig:slice}
    \end{figure}
    
    \begin{figure}[htbp]
    \centering
    \includegraphics[width=\textwidth]{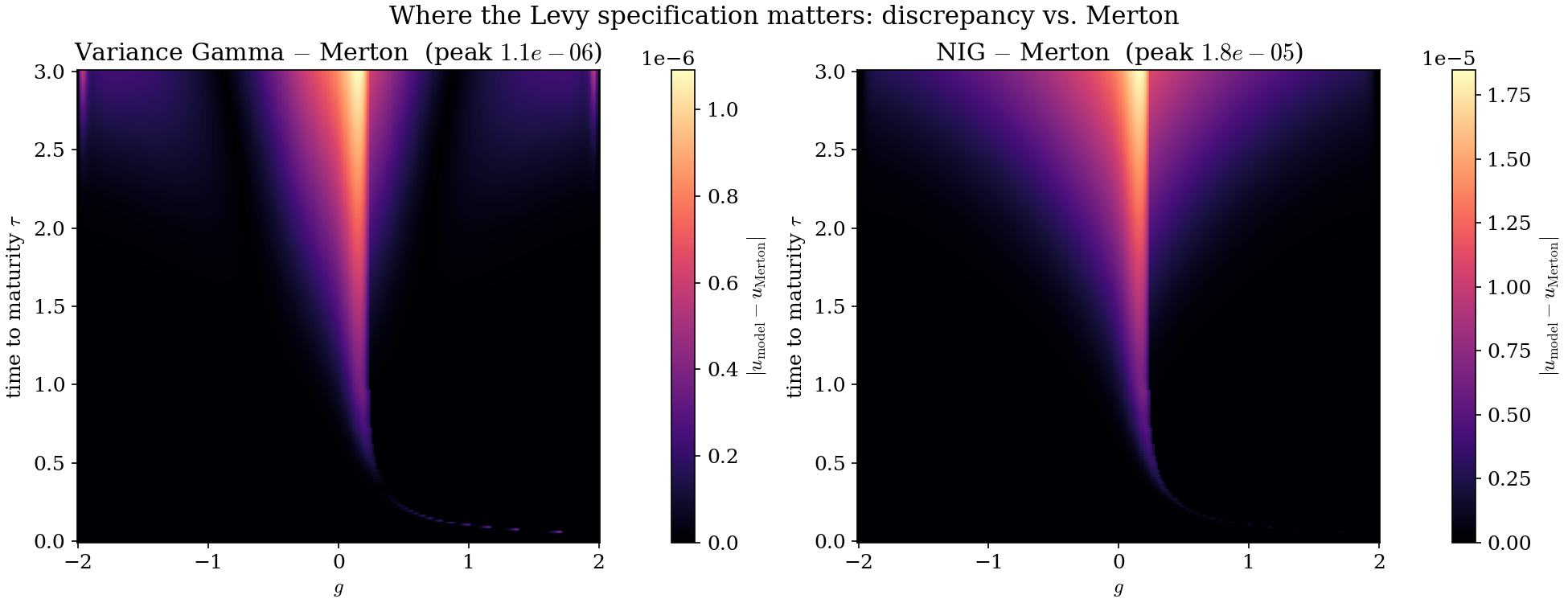}
    \caption{Pointwise discrepancy $|u_{\mathrm{model}}-u_{\mathrm{Merton}}|$ over
    the $(g,\tau)$ grid for VG (left) and NIG (right). Note the colorbar scales
    ($10^{-6}$ and $10^{-5}$): the differences are localized near the free boundary
    and lie far below the discretization accuracy of the scheme.}
    \label{fig:heatmap}
    \end{figure}
    
    \begin{table}[htbp]
    \centering
    \begin{tabular}{lc}
    \toprule
    \textbf{Model pair} & $\displaystyle\max_{g,t}\bigl|u_i-u_j\bigr|$ \\ \midrule
    Merton vs.\ VG   & $1.1\times10^{-6}$ \\
    Merton vs.\ NIG  & $1.8\times10^{-5}$ \\
    Merton vs.\ Kou  & $3.0\times10^{-6}$ \\
    VG vs.\ NIG      & $2.0\times10^{-5}$ \\
    VG vs.\ Kou      & $4.1\times10^{-6}$ \\
    NIG vs.\ Kou     & $1.6\times10^{-5}$ \\ \bottomrule
    \end{tabular}
    \caption{Maximum pointwise discrepancy between the bubble surfaces over the full
    $(g,t)$ grid. All differences lie far below the discretization accuracy of the
    scheme: the discretization biases are common to the four models and cancel in
    the differences.}
    \label{tab:diffs}
    \end{table}
    
    \begin{figure}[htbp]
    \centering
    \includegraphics[width=\textwidth]{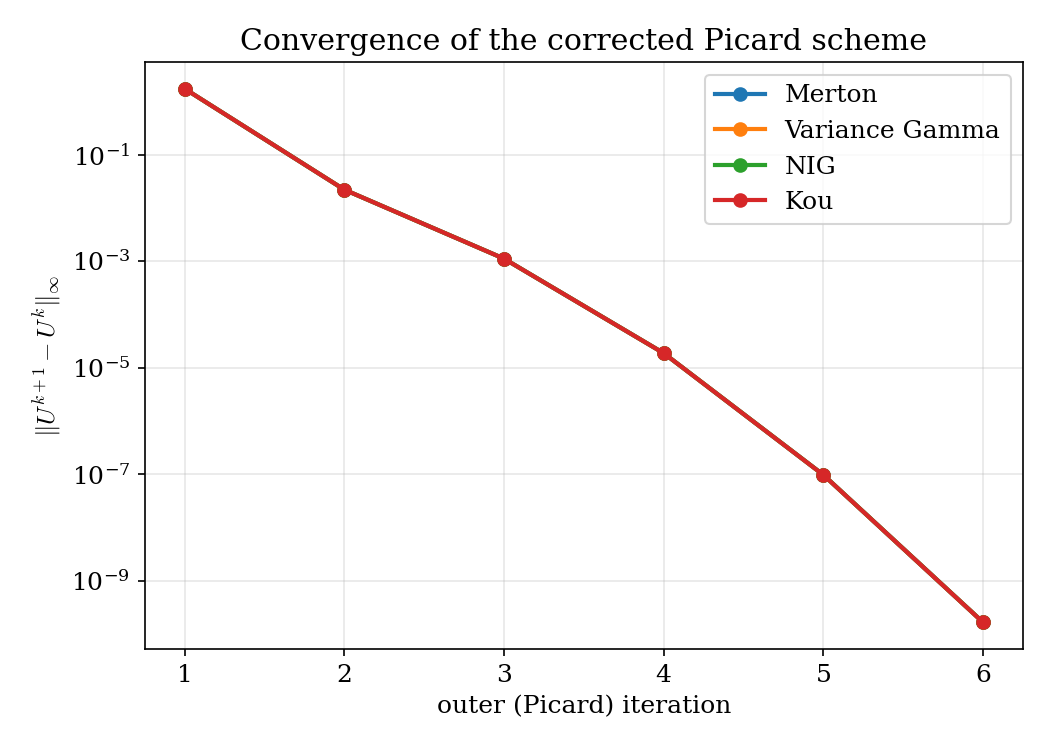}
    \caption{Convergence of the outer Picard iteration: the global update
    $\lVert U^{k+1}-U^{k}\rVert_\infty$ falls below $10^{-9}$ within six
    iterations for every model.}
    \label{fig:convergence}
    \end{figure}

    \subsection{Discussion}
    
    The experiments show that, at this calibration, the bubble value is almost
    insensitive to the fine structure of the
    driving L\'evy measure---finite activity (Merton, Kou) versus infinite activity
    (VG, NIG)---even though the measures
    differ sharply (Figure~\ref{fig:densities}). There are two reasons for this. First, the premium is set by the resale obstacle and the mean-reverting
    diffusion: the peak value is the Dirichlet datum imposed at the truncation
    boundary,
    $u_h(g_{\max},0)=\psi(g_{\max},0)=g_{\max}\,\bar\alpha(T)-c\approx1.7232$,
    independent
    of the jump law, and with $\gamma=c_J\,g\,z$ the discretized jump operator is
    uniformly small relative to the local part
    ($\Delta t\,\lVert J\rVert_\infty$ ranges from $1.2\times10^{-2}$ for VG to
    $6.7\times10^{-2}$ for NIG, against
    $\Delta t\,\lVert D\rVert_\infty\approx 15$). Second, because the jumps scale
    with the state, a post-jump value $g(1+c_Jz)$ leaving the window
    $[g_{\min},g_{\max}]$ is truncated by the quadrature, so the jump contribution
    saturates rather than growing without bound.
    
    The growth condition \eqref{eq:growth_cond} is needed for the existence and uniqueness in Theorem~\ref{thm:exist-pide},
    but the numerical scheme remains well-behaved outside this regime. When the parameters
    violate the condition (large $c$ or small $\eta_i$), the PSOR/Picard solver still
    converges, though more slowly, with more outer iterations.
    The computed functions stay non-negative and obstacle-compliant,
    which suggests that the regularization of the implicit scheme extends beyond the range covered by the theory.
    
    Two modeling caveats accompany this conclusion. The symmetric small-jump
    exclusion $(-\epsilon,\epsilon)$ removes the accumulation of infinitely many
    small jumps that distinguishes VG and NIG from finite-activity models: at
    $\epsilon=0.05$ it discards about $42\%$ of the NIG jump variance
    $\int z^2\,\nu(dz)$ (but only $2.7\%$ for VG and $3.1\%$ for Merton), so the
    infinite-activity models are effectively compared through their truncated
    measures. Resolving the genuine infinite-activity regime would require
    compensating the excluded mass by the state-dependent diffusion correction
    $\sigma_\epsilon^2(g)=\int_{|z|<\epsilon}\gamma(g,z)^2\,\nu(dz)$ and
    refining $\epsilon\to0$. The
    jump law would also exert a visible influence once $c_J$ or the intensity is
    increased so that $\Delta t\,\lVert J\rVert$ becomes comparable to the local
    operator. Within the present calibration, however, the computed value function
    is non-negative, vanishes as $g\to g_{\min}$, grows linearly toward $g_{\max}$
    in agreement with the $\mathcal{O}(|g|)$ uniqueness class of
    Section~\ref{sec:existence}, and reproduces the obstacle-driven free boundary,
    confirming the theoretical predictions and the robustness of the PSOR/Picard
    scheme across both finite- and infinite-activity L\'evy structures.

\bibliographystyle{alpha} 
\bibliography{Bubble-Levy}

\end{document}